\documentclass[a4paper,11pt]{article}
\usepackage{MyStyle}
\usepackage{eatree}
\usepackage{upgreek}

\SetSymbolFont{stmry}{bold}{U}{stmry}{m}{n}

\tikzgraphsset{
  implies/.style={edges={double, double equal sign distance, -implies}},
}
\usetikzlibrary{graphs}
\usetikzlibrary{calc}

\newcommand{\Circled}[1]{\mathbf{#1}}

\DeclareMathOperator{\Diff}{Diff}
\DeclareMathOperator{\leftiso}{\LL\OO}
\DeclareMathOperator{\rightiso}{\RR\OO}

\allowdisplaybreaks

\newtheorem{theorem}{Theorem}[section]
\newtheorem{definition}[theorem]{Definition}
\newtheorem*{definition*}{Definition}

\newtheorem{proposition}[theorem]{Proposition}
\newtheorem{lemma}[theorem]{Lemma}
\newtheorem{remark}[theorem]{Remark}
\newtheorem*{remark*}{Remark}

\newtheorem*{remarks*}{Remarks}

\newtheorem*{notation*}{Notation}

\newtheorem*{ex*}{Example}

\newtheorem*{exs*}{Examples}

\newtheorem*{app*}{Application}
\newtheorem{conjecture*}{Conjecture}
\def\ts{\thinspace}

\title{
The universal equivariance properties of exotic aromatic B-series
}

\author{
Adrien Laurent\textsuperscript{1} and Hans Munthe-Kaas\textsuperscript{2}
}

\begin{document}
\footnotetext[1]{
%Department of Mathematics, University of Bergen, Bergen, Norway.
%Adrien.Laurent@uib.no.}
Univ Rennes, INRIA (Research team MINGuS), IRMAR (CNRS UMR 6625) and ENS Rennes, France.
Adrien.Laurent@inria.fr.}
\footnotetext[2]{
Department of Mathematics and Statistics, UiT -- The Arctic University of Norway, Tromsø, Norway. Hans.Munthe-Kaas@uit.no.}

\maketitle

\begin{abstract}
The exotic aromatic Butcher series were originally introduced for the calculation of order conditions for the high order numerical integration of ergodic stochastic differential equations in~$\R^d$ and on manifolds.
We prove in this paper that exotic aromatic B-series satisfy a universal geometric property, namely that they are characterised by locality and equivariance with respect to orthogonal changes of coordinates.
This characterisation confirms that exotic aromatic B-series are a fundamental geometric object that naturally generalises aromatic B-series and B-series, as they share similar equivariance properties.
In addition, we provide a classification of the main subsets of the exotic aromatic B-series, in particular the exotic B-series, using different equivariance properties.
Along the analysis, we present a generalised definition of exotic aromatic trees, dual vector fields, and we explore the impact of degeneracies on the classification.

\smallskip

\smallskip


%\bigskip
%\noindent
%Communicated by Christian Lubich 
\end{abstract}

\section{Introduction}
\label{section:Introduction}

Consider the ordinary differential equation
\begin{equation}
\label{equation:ODE}
y'(t)=f(y(t)),\quad y(0)=y_0,
\end{equation}
where~$f\colon \R^d\rightarrow \R^d$ is a Lipschitz vector field and~$y_0\in \R^d$, and a one-step integrator for solving~\eqref{equation:ODE} of the form
\begin{equation}
\label{equation:one-step_integrator}
y_{n+1}=\Phi_h(y_n),
\end{equation}
where~$h$ is the timestep of the method.
Following the backward error analysis idea~\cite{Hairer06gni}, in order to study the properties of the integrator~\eqref{equation:one-step_integrator} (preservation of invariants or measures, order, behaviour in long-time,\dots), it proves convenient to rewrite the scheme as the exact solution of a modified ODE
\[
%\label{equation:modified_ODE}
\tilde{y}'(t)=\tilde{f}_h(\tilde{y}(t)).
\]
For large classes of integrators, such as Runge-Kutta methods, the modified vector field~$\tilde{f}_h$ can be expressed as a formal Taylor series in~$f$ and its partial derivatives, called a B-series~\cite{Hairer06gni, Chartier10aso, Calaque11tih} (see also \cite{Iserles00lgm, Lundervold11hao, Lundervold13bea, Rahm22aoa} in the context of Lie-group methods).
The paper~\cite{McLachlan16bsm} presents universal geometric conditions on~$\tilde{f}_h$ to ensure that it can be written as a B-series. More precisely, any smooth local map that is invariant under affine changes of coordinates (including affine maps between vector spaces of different dimensions, see Definition~\ref{definition:adim_iso_equiv}) can formally be written as a B-series.

Originally introduced in~\cite{Butcher72aat,Hairer74otb}, Butcher series proved to be a powerful tool for the construction of numerical integrators with high order of accuracy and for the numerical preservation of geometric properties (see, for instance, the textbooks~\cite{Hairer06gni,Butcher16nmf,Butcher21bsa} and the review~\cite{McLachlan17bsa}).
The papers~\cite{Chartier07pfi,Iserles07bsm} introduced simultaneously an extension of B-series called aromatic B-series for the study of volume preserving integrators (see also the recent works~\cite{Bogfjellmo19aso,Bogfjellmo22uat,Lejay22cgr,Laurent23tab,Laurent23tld,Bronasco22cef}).
In this context, one is interested in finding methods~\eqref{equation:one-step_integrator} satisfying~$\Div(\tilde{f}_h)=0$. Aromatic B-series are natural objects in this context as the divergence of a standard B-series rewrites conveniently with aromatic B-series.
One then wonders whether B-series and aromatic B-series are merely tools for manipulating tedious Taylor expansions or natural far-reaching algebraic objects.
This question is answered in~\cite{McLachlan16bsm,MuntheKaas16abs} where universal geometric characterisations of B-series and aromatic B-series are given (see also~\cite{Markl08git,McLachlan17bsa}).

In the context of stochastic differential equations (SDEs), it is known that backward error analysis does not generalise straightforwardly~\cite{Shardlow06mef,Zygalakis11ote}. However, there exists a similar concept for ergodic SDEs.
Consider overdamped Langevin dynamics with Stratonovich noise of the form
\begin{equation}
\label{equation:general_SDE}
dY(t)=\Pi_\MM(Y(t))f(Y(t))dt+\Pi_\MM(Y(t)) \circ dW(t),\quad Y(0)=Y_0,
\end{equation}
where~$f\colon \R^d\rightarrow \R^d$ is a Lipschitz vector field,~$Y_0\in \R^d$,~$\Pi_\MM(x)$ is the orthogonal projection on the tangent bundle of~$\MM$ at the point~$x\in \MM$ (note that~$\Pi_\MM(x)=I_d$ if~$\MM=\R^d$), and~$W$ is a standard~$d$-dimensional Brownian motion in~$\R^d$ on a probability space equipped with a filtration and fulfilling the usual assumptions.
Under a growth assumption on~$f$, the solution of~\eqref{equation:general_SDE} is ergodic, that is, it follows a deterministic distribution, called the invariant measure, in long time~\cite{Faou09csd,Debussche12wbe,Abdulle14hon}.
In~\cite{Laurent20eab,Laurent21ocf,Laurent21ata}, an extension of the aromatic B-series, called exotic aromatic B-series, is introduced to conveniently write Taylor expansions (called Talay-Tubaro expansions~\cite{Talay90eot} in this context) of the solutions of~\eqref{equation:general_SDE} and to build high-order approximations of the invariant measure of~\eqref{equation:general_SDE}, with applications in molecular dynamics~\cite{Lelievre10fec}.
The main idea is to introduce two new types of edges to represent the Laplacian and the scalar product.
Ergodic integrators for solving~\eqref{equation:general_SDE} have an invariant measure that can be written as the invariant measure of an exact problem of the form~\eqref{equation:general_SDE} with a modified vector field~$\tilde{f}_h$ that typically has the form
\begin{equation}
\label{equation:modified_vf_example}
h\tilde{f}_h=h f^i\partial_i
+h^2[c_1 f^i_{jj}+c_2 f^i_j f^j +c_3 f^j f^j f^i]\partial_i
+h^3[c_4 f^i_{jjkk}+c_5 f^j_k f^i_{jk}+c_6 f^k f^k_j f^j f^i]\partial_i
+\dots
\end{equation}
where~$\partial_i$ is the vector basis of~$\R^d$, the~$c_n$ are constants, and each term is summed on all involved indices, in the spirit of the Einstein summation notation.
Note that the expansion~\eqref{equation:modified_vf_example} is a linear combination of monomials in the components~$f^i$ and their partial derivatives, where we use pairs of indices. Note also that the power of~$h$ associated to a monomial is not given by the number of occurrences of~$f$, in opposition to the deterministic context.
For the integrators presented in~\cite{Laurent20eab,Laurent21ocf} for solving~\eqref{equation:general_SDE}, the modified vector field can be expressed as an exotic B-series~\cite{Bronasco22cef} in~$\R^d$ (see examples in~\cite[Sec.\ts 5.1]{Laurent20eab}) and as a partitioned exotic aromatic B-series on manifolds at least for the first orders~\cite{Laurent21ocf,Bronasco22cef}.
Moreover, the Talay-Tubaro expansions presented in~\cite{Laurent20eab,Laurent21ocf} are exotic aromatic S-series~\cite{Bronasco22ebs}.

For the high-order approximation of~\eqref{equation:general_SDE}, the number of terms in the Taylor expansions grows quickly, which further motivates the use of exotic aromatic B-series for the study of integrators for solving SDEs (see, for instance, the order two expansion in~\cite[App.\ts D]{Laurent21ocf}).
Then, a natural question is the following: are exotic aromatic B-series just a technical tool used for carrying out tedious calculation, or are they fundamental objects satisfying similar geometric properties as B-series and aromatic B-series?
In this paper, we show that the exotic aromatic B-series satisfy a universal equivariance property, which justifies that the exotic aromatic B-series formalism is a natural extension of B-series and aromatic B-series.
We extend this result to characterise subsets of the exotic aromatic B-series, such as the exotic B-series.
This work also allows us to give a more general definition of exotic aromatic B-series, free of the degeneracies tied to the numerical context.

The article is organized as follows.
We present in Section~\ref{section:preliminaries} the definition of the geometric properties used in the characterisation, the general definition of the exotic aromatic B-series, and the main results of the paper.
%We also present the main result of this work, that is, that all local right-orthogonal-equivariant maps can be rewritten as exotic aromatic B-series.
The characterisation of exotic aromatic B-series is proven in Section~\ref{section:proof}, while we derive a finer classification of the main subsets of exotic aromatic B-series in Section~\ref{section:proof_strong_equiv}.
We give outlooks on future works in Section~\ref{section:future_work}.

\section{Preliminaries and main results}
\label{section:preliminaries}

This section is devoted to the definition of locality, equivariance and decoupling, to the new general definition of exotic aromatic trees and their associated elementary differential, and to the main results of the paper (Subsection~\ref{section:main_result}).

\subsection{Locality and equivariance}
\label{section:def_geom_properties}

We define the geometric properties used in the characterisation of exotic aromatic B-series. A natural property of modified vector fields is locality.
%Large classes of one-step integrators~\eqref{equation:one-step_integrator} are local, such as Runge-Kutta methods or exponential integrators.
\begin{definition}
Let~$\mathfrak{X}(\R^d)$ the set of smooth vector fields on~$\R^d$, a map~$\varphi_d\colon\mathfrak{X}(\R^d)\rightarrow\mathfrak{X}(\R^d)$ is local if
\[\supp(\varphi_d(f))\subset\supp(f), \quad \supp(f)=\overline{\{x\in\R^d,f(x)\neq 0\}}.\]
%We denote the vector space of smooth local maps by~$\overset{\bullet}{\CC}^\infty(\mathfrak{X}(\R^d),\mathfrak{X}(\R^d))$.
\end{definition}

In~\cite{MuntheKaas16abs}, the aromatic B-series are characterised by locality and a property of equivariance.
Let~$G$ be a finite dimensional Lie subgroup of the set of diffeomorphisms~$\Diff(\R^d)$ on~$\R^d$.
In this work, we consider groups of the form~$G=H\ltimes \R^d$, where~$H$ is a subgroup of~$\GL_d(\R)$ called the isotropy group.
An element~$g=(A,b)\in G$ acts on a vector field~$f\in \mathfrak{X}(\R^d)$ by
\[(g\cdot f)(x)=Af(A^{-1}(x-b)).\]
The~$G$-equivariance is the compatibility with the action of~$G$ on vector fields.
\begin{definition}
A map~$\varphi_d\colon\mathfrak{X}(\R^d)\rightarrow\mathfrak{X}(\R^d)$ is~$G$-equivariant if
\[\varphi_d(g\cdot f)=g\cdot \varphi_d(f), \quad g\in G, \quad f\in \mathfrak{X}(\R^d).\]
%The vector space of smooth~$G$-equivariant maps is denoted~$\CC^\infty_G(\mathfrak{X}(\R^d),\mathfrak{X}(\R^d))$.
\end{definition}
For the sake of simplicity, we call~$\GL$-equivariance (respectively orthogonal-equivariance) the~$H\ltimes \R^d$-equivariance properties with~$H=\GL_d(\R)$ (respectively~$H=\Or_d(\R)$).
Note that the locality and~$G$-equivariance are properties in fixed dimension.
The first main result of this paper is the characterization of exotic aromatic B-series with orthogonal-equivariance and locality.

The second goal of this work is to characterize the subsets of exotic aromatic B-series, in particular the exotic B-series. In this context, the dimension~$d\geq 0$ of the problem plays an important role, so that we rely on sequences of maps~$\varphi=(\varphi_d\colon \mathfrak{X}(\R^d)\rightarrow\mathfrak{X}(\R^d))_d$ indexed by the dimension~$d$.
Such a sequence is local (respectively~$G$-equivariant) if~$\varphi_d$ is local (respectively~$G$-equivariant) for all~$d$.
To observe the interactions between the dimensions, the notion of equivariance is extended to affine transformations~\cite{McLachlan16bsm}, and we describe such a property as a categorical equivariance property in the following.
%We shall characterise the different B-series with the following \moda{categorical equivariance} properties.
\begin{definition}
\label{definition:adim_iso_equiv}
Let~$\Aff=(\Aff_{d_1,d_2})_{d_1,d_2}$ be the category of affine transformations with
\[\Aff_{d_1,d_2}=\{a\colon \R^{d_1}\rightarrow \R^{d_2}, \ a(x)=Ax+b, \ (A,b)\in \R^{d_2\times d_1} \times \R^{d_2}\},\quad d_1, d_2\geq 0.\]
Let~$\HH=(\HH_{d_1,d_2})_{d_1,d_2}$ be a subcategory of~$\Aff$.
A sequence of smooth maps~$\varphi=(\varphi_d\colon \mathfrak{X}(\R^d)\rightarrow\mathfrak{X}(\R^d))_d$ is equivariant with respect to~$\HH$ if for all~$d_1$,~$d_2\geq 0$, for all~$a(x)=Ax+b\in \HH_{d_1,d_2}$ and~$x\in \R^{d_1}$,~$\varphi$ satisfies
\[
f_2(a(x))=Af_1(x) \ \Rightarrow \ \varphi_{d_2}(f_2)(a(x))=A\varphi_{d_1}(f_1)(x), \quad f_1\in \mathfrak{X}(\R^{d_1}), \quad f_2\in \mathfrak{X}(\R^{d_2}).
\]
A sequence of smooth maps~$\varphi=(\varphi_d)_d$ is affine-equivariant if it is equivariant with respect to~$\Aff$.
\end{definition}

%The different subsets of~$\Aff_{d_1,d_2}$ that we consider are associated to the classical homogeneous spaces corresponding to the Lie-group~$H=\Or_d(\R)$:
As the Taylor expansions that appear in the study of stochastic dynamics are typically orthogonal-equivariant, we consider the equivariance properties with respect to subcategories of~$\Aff$ on which the orthogonal group acts naturally.
\begin{definition}
Let the categories of left-orthogonal and right-orthogonal transformations
\begin{align}
\label{equation:Stiefel}
\leftiso_{d_1,d_2}&=\{a(x)=Ax+b\in\Aff_{d_1,d_2}, \ A^TA=I_{d_1}\},\\
\label{equation:Grassmann}
\rightiso_{d_1,d_2}&=\{a(x)=Ax+b\in\Aff_{d_1,d_2}, \ AA^T=I_{d_2}\}.
\end{align}
The sequence of smooth maps~$\varphi=(\varphi_d)_d$ is left-orthogonal-equivariant (respectively right-orthogonal-equivariant) if it is equivariant with respect to all affine transformations in~$\leftiso$ (respectively~$\rightiso$).
A sequence~$\varphi=(\varphi_d)_d$ that is both left and right-orthogonal-equivariant is called semi-orthogonal-equivariant.
\end{definition}
Note that the categories used in this paper all have a structure of semigroupoids, and~$\leftiso$ is associated to the Stiefel manifolds.

The elementary differentials associated to standard B-series have the property to keep decoupled systems decoupled. Some exotic aromatic B-series also satisfy this property, which motivates the following definition.
Similarly to~\cite{McLachlan16bsm}, for~$f_{1}\in \mathfrak{X}(\R^{d_1})$ and~$f_2\in \mathfrak{X}(\R^{d_2})$, we use the notation
\[h=f_1\oplus f_2\in \mathfrak{X}(\R^{d_1+d_2}), \quad h(x,y)=(f_1(x),f_2(y)).\]
\begin{definition}
A sequence~$\varphi=(\varphi_d\colon \mathfrak{X}(\R^d)\rightarrow\mathfrak{X}(\R^d))_d$ is decoupling if for all~$f_{1}\in \mathfrak{X}(\R^{d_1})$ and~$f_2\in \mathfrak{X}(\R^{d_2})$,~$\varphi$ satisfies
\[\varphi_{d_1+d_2}(f_1\oplus f_2)=\varphi_{d_1}(f_1)\oplus \varphi_{d_2}(f_2),\]
that is, for all~$x\in \R^{d_1}$,~$y\in \R^{d_2}$,
\[\varphi_{d_1+d_2}((f_1(x), f_2(y)))=(\varphi_{d_1}(f_1)(x),\varphi_{d_2}(f_2)(y)).\]
A sequence~$\varphi$ is trivially decoupling if for all~$f\in \mathfrak{X}(\R^{d_1})$,~$\varphi$ satisfies
\[\varphi_{d_1+d_2}(f\oplus 0)=\varphi_{d_1}(f)\oplus 0.\]
\end{definition}

%Note that locality and decouplingness imply~$0$-preservingness.
The left-orthogonal-equivariance and right-orthogonal-equivariance properties are stronger than the orthogonal-equivariance in the following sense. The proof is omitted as it is nearly identical to~\cite[Lem.\ts 4.2 and 6.1]{McLachlan16bsm}.
\begin{proposition}[\cite{McLachlan16bsm}]
Let~$\varphi=(\varphi_d\colon \mathfrak{X}(\R^d)\rightarrow\mathfrak{X}(\R^d))_d$ be a sequence of smooth maps.
If~$\varphi$ is left-orthogonal-equivariant, then~$\varphi$ is local, orthogonal-equivariant, and trivially decoupling.
If~$\varphi$ is right-orthogonal-equivariant, then~$\varphi$ is orthogonal-equivariant, trivially decoupling, and decoupling.
\end{proposition}

We summarise the links between the different geometric properties in the following graph for clarity.
\begin{center}
\begin{tikzpicture}
	\node[draw] (1) {semi-orth.-equivariance};
	\node[draw,anchor=west] (2) at ($(1)+(3.2cm,0.5cm)$) {left-orth.-equivariance};
	\node[draw,anchor=west] (3) at ($(1)+(3.2cm,-0.5cm)$) {right-orth.-equivariance};
	\node[draw,anchor=west] (4) at ($(2)+(3.23cm,1cm)$) {locality};
	\node[draw,anchor=west] (5) at ($(3)+(3.1cm,-1cm)$) {decoupling};
	\node[draw,anchor=west] (6) at ($(2)+(3.23cm,0cm)$) {orth.-equivariance};
	\node[draw,anchor=west] (7) at ($(3)+(3.1cm,0cm)$) {trivially decoupling};
	\draw[->] (1.east) -- (2.west);
	\draw[->] (1.east) -- (3.west);
	\draw[->] (2.east) -- (4.south west);
	\draw[->] (3.east) -- (5.north west);
	\draw[->] (2.east) -- (6.west);
	\draw[->] (3.east) -- (6.south west);
	\draw[->] (2.east) -- (7.north west);
	\draw[->] (3.east) -- (7.west);
\end{tikzpicture}
\end{center}

\subsection{Exotic aromatic trees}

Introduced originally in~\cite{Laurent20eab,Laurent21ocf} for the calculation of order conditions for the approximation of ergodic stochastic differential equations, the exotic aromatic trees are an extension of aromatic trees that involves two new kind of edges: lianas and stolons.
The definition we give in the present paper is a generalization of the one originally presented in~\cite{Laurent20eab,Laurent21ocf} (see also~\cite{Bronasco22ebs,Bronasco22cef}).
It reduces to the same definition under a regularity assumption discussed in Section~\ref{section:degeneracies_2}.
We choose an approach based on permutations as in~\cite{MuntheKaas16abs} (see also~\cite{Bogfjellmo19aso,Laurent23tab}).

In the literature, Butcher trees and aromatic trees are defined as directed graphs. Such graphs are encoded by a set of vertices~$V$, a set of arrows~$A$, and two maps~$\sigma$ and~$\tau$ that give the source and target of every arrow. A node formally represents an evaluation while an arrow represents a derivation.
For instance, the Butcher tree~$\eatree{1}{000}{1}$ represents~$f=f^i \partial_i$ and the aromatic tree~$\eatree{11}{100}{1}$ stands for~$\Div(f)f=f^j_j f^i \partial_i$.
In the stochastic context, the Laplacian operator and the scalar product appear in the Taylor expansions and they respectively act as double derivations and double evaluations.
To account for these new operations, we use a generalisation of directed graphs that allows arrows to be sources of arrows and nodes to be sources of other nodes. The Laplacian is then represented with a double arrow (i.e., a double derivation), while the scalar product is given by two nodes identified through the source map~$\sigma$.
\begin{definition}
\label{definition:exotic_aromatic_trees}
We consider graphs of the form~$(V,\Circled{A_0},\sigma,\tau)$ with~$V$ a finite set of vertices indexed from 1 to~$\abs{V}$,~$\Circled{A_0}$ a finite set of arrows indexed from~$\Circled{0}$ to~$\abs{\Circled{A}}$, where~$\Circled{A}=\Circled{A_0} \setminus \{\Circled{0}\}$. The source map~$\sigma\colon V\cup \Circled{A_0}\rightarrow V\cup \Circled{A_0}$ has no fixed points and satisfies~$\sigma\circ\sigma=\text{id}$, and the target map is~$\tau\colon \Circled{A}\rightarrow V$.
%The map~$\tau\colon \Circled{A}\rightarrow V$ is the target map.
%The source map is a permutation~$\sigma\colon V\cup \Circled{A_0}\rightarrow V\cup \Circled{A_0}$ that has no fixed points and satisfies~$\sigma\circ\sigma=\text{id}$.
Two such graphs are equivalent if there exists a bijection between their sets of nodes and arrows that are compatible with the source and target maps.
An exotic aromatic tree is an equivalence class of such graphs.
We denote~$\Gamma$ the set of exotic aromatic trees.
\end{definition}

Let us add some vocabulary and graphical rules to describe and draw exotic aromatic trees.
%Definition~\ref{definition:exotic_aromatic_trees} differs from standard definitions of directed graphs as the source map usually sends arrows to nodes. The extension presented here allows arrows to be sources of arrows and vertices to be sources of vertices.
If~$\sigma(\Circled{a_1})=\Circled{a_2}$, we say that the unordered pair~$(\Circled{a_1},\Circled{a_2})$ is a liana and we represent it with a dashed edge between the two nodes~$\tau(\Circled{a_1})$ and~$\tau(\Circled{a_2})$, that can be identical.
If~$\sigma(v_1)=v_2$, we call the unordered pair~$(v_1,v_2)$ a stolon and we draw it with a double edge between~$v_1$ and~$v_2$.
The set of lianas is denoted~$L$ and the set of stolons is~$S$.
An exotic aromatic tree without lianas and stolons is called an aromatic tree, an extension of standard trees allowing for loops.
A loop is a list of nodes~$(v_1,\dots,v_K)$ such that there is a standard edge linking~$v_1$ to~$v_2$, \dots,~$v_K$ to~$v_{1}$ (also called K-loop in~\cite{Iserles07bsm}).
Note that an exotic aromatic tree is an aromatic tree if and only if~$\sigma(\Circled{A_0})=V$. In this case, Definition~\ref{definition:exotic_aromatic_trees} reduces to an equivalent definition of the one in~\cite{MuntheKaas16abs}.
We refer to Table~\ref{table:example_EAF} for examples.

\begin{ex*}
%\label{example:ex_ex_aro_tree}
Let the exotic aromatic tree~$\gamma=(V,\Circled{A_0},\sigma,\tau)$ with the nodes~$V=\{1,2,3,4\}$, the arrows~$\Circled{A_0}=\{\Circled{0},\Circled{1},\Circled{2},\Circled{3}\}$, and the following source and target maps
\[\sigma=(\Circled{0},1)(\Circled{1},\Circled{2})(2,3)(\Circled{3},4),\quad
\tau=(2,3,4),\]
where we use the notation~$
\tau=(\tau(\Circled{1}),\dots,\tau(\Circled{\abs{\Circled{A}}}))$.
The tree~$\gamma$ has one loop~$(4)$, one liana~$(\Circled{1},\Circled{2})$, and one stolon~$(2,3)$. The associated graph is the following, where we detail the vertices and arrows for clarity.
\[
\, \includegraphics[scale=0.75]{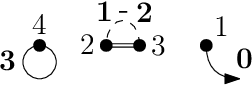} \,
\]
\end{ex*}

An exotic aromatic tree has a unique root~$r$ defined the following way.
If~$\sigma(\Circled{0}) \in V$,~$r=\sigma(\Circled{0})$ is the root and~$\Circled{0}$ is called a ghost arrow.
If~$\sigma(\Circled{0}) \in \Circled{A}$, the liana~$r=(\Circled{0},\sigma(\Circled{0}))\in L$ is the root and is called a ghost liana.
We draw the ghost arrow (respectively ghost liana) on the graphical representations of exotic aromatic trees as an edge (respectively dashed edge) with one end left unattached.
In the aromatic context, the ghost arrow is usually omitted on the graphical representation, hence the name.

We define a notion of connectedness on exotic aromatic trees by saying that two nodes or arrows~$x$,~$y\in V\cup \Circled{A_0}$ are neighbours if~$\sigma(x)=y$ or~$\tau(x)=y$ or~$\tau(y)=x$.
The connected components without the root are called aromas and a finite unordered collection of aromas is a multi-aroma.
The connected component with the root is a connected exotic aromatic tree. We denote~$\Gamma_c$ the set of connected exotic aromatic trees and~$\Gamma^0$ the set of multi-aromas, also represented as equivalence classes of graphs~$(V,\Circled{A},\sigma,\tau)$ without the arrow~$\Circled{0}$. The aromas are gathered in~$\Gamma^0_c$. An exotic aromatic tree decomposes into a finite number of aromas and one connected exotic aromatic tree. This notion of connectedness is a strong motivation to understand the exotic aromatic trees as graphs and not as trees as done beforehand in the literature.
If there are no lianas and no stolons, we find the standard definition of aromas and rooted trees in the context of aromatic trees.

We define an exotic tree (respectively a stolonic tree) as an exotic aromatic tree that reduces to a standard Butcher tree when removing all the lianas (respectively by removing all the stolons).
Note that there is a difference between the notions of exotic trees and connected exotic aromatic trees. A connected exotic aromatic tree does not reduce to a tree in general when removing the lianas.

\begin{ex*}
Let the following exotic aromatic trees
\[\gamma_1=\eatree{11}{100}{1}, \quad
\gamma_2=\eatree{011}{110}{1}, \quad
\gamma_3=\eatree{011}{010}{1}.\]
The exotic aromatic tree~$\gamma_1$ is a disconnected aromatic tree with one aroma. The graph~$\gamma_2$ is connected, but is not an exotic tree as removing the lianas of~$\gamma_2$ yields~$\gamma_1$, which is not a tree.
On the other hand,~$\gamma_3$ is an exotic tree.
\end{ex*}

We denote the set of nodes that are the target of~$j$ arrows by~$V_j$. For a given aromatic tree~$\gamma$, we define its composition~$\kappa\colon \N\rightarrow\N$ by~$\kappa(j)=\abs{V_j}$, and its derived composition by~$\boldsymbol{\kappa}'(j)=j\kappa(j)$. A straightforward observation yields that the cardinals of~$V$ and~$\Circled{A}$ satisfy~$\abs{V}=\abs{\kappa}$ and~$\abs{\Circled{A}}=\abs{\boldsymbol{\kappa}'}$, where~$\abs{\kappa}=\kappa(0)+\kappa(1)+\dots$
%$\abs{\kappa}$ is called the order of the aromatic tree~$\gamma$.
%For simplicity, we index the nodes and arrows of an aromatic tree by natural numbers:
%\[N=\{1,\dots,\abs{\kappa}\}, \quad A=\{\Circled{1},\dots,\Circled{\abs{\boldsymbol{\kappa}'}}\},\]
%and the composition and target maps by
%\[
%\tau=(\tau(\Circled{1}),\dots,\tau(\Circled{\abs{\boldsymbol{\kappa}'}})),\quad
%\kappa=(\kappa(0),\kappa(1),\dots),\quad
%\boldsymbol{\kappa}'=(\boldsymbol{\kappa}'(0),\boldsymbol{\kappa}'(1),\dots).
%\]
%By definition, there is one more node than arrows, so that
%$\abs{\kappa}$ and~$\abs{\boldsymbol{\kappa}'}$ satisfy
%\begin{equation}
%\label{equation:equality_kappa_aromatic}
%\abs{\kappa}=\abs{\boldsymbol{\kappa}'}+1.
%\end{equation}
%We define an extra arrow, that we call the ghost arrow~$\Circled{0}$. The source of this arrow is the root~$\sigma(\Circled{0})=r$. The source map~$\sigma$ can be prolonged on~$\Circled{A}=A\cup\{\Circled{0}\}$, which makes it a bijection between~$N$ and~$\Circled{A}$.
%We further extend~$\sigma$ on~$N\cup \Circled{A}$ by writing~$\sigma(v)=\Circled{A_0}$ iff~$v=\sigma(\Circled{A_0})$. This makes~$\sigma$ into a permutation of the elements of~$N\cup \Circled{A}$.
%Note that such a permutation~$\sigma$ satisfies~$\sigma\circ\sigma=\text{id}$. Thus, it can be decomposed in a composition of transpositions with disjoint supports.
We write~$\Gamma_\kappa$ the set of exotic aromatic trees with composition~$\kappa$ and~$\Gamma_m$ the set of exotic aromatic trees such that~$\abs{\kappa}=m$. Note that~$\Gamma_\kappa$ is finite while~$\Gamma_m$ is infinite for all~$m$.
Contrary to the case of Butcher trees and aromatic trees, the order of an exotic aromatic tree is not given by the number of its nodes~$\abs{\kappa}$, but by the number of transpositions in the source map~$\sigma$. We shall see in Definition~\ref{definition:elementary_differential} that the order represents the number of pairs of indices appearing in the associated elementary differential (see, for instance, in~\eqref{equation:modified_vf_example}). We follow the definition of the order of an exotic aromatic tree given in~\cite{Laurent20eab,Laurent21ocf}.
% We recall that~$\abs{\kappa}=\abs{V}$ and~$\abs{\boldsymbol{\kappa}'}=\abs{A}$.
\begin{lemma}
\label{lemma:def_order}
Define the order~$\abs{\gamma}$ of an exotic aromatic tree~$\gamma\in \Gamma_\kappa$ by~$\abs{\gamma}=\abs{V}+\abs{L}-\abs{S}$.
Then the following identity holds
\begin{equation}
\label{equation:equality_kappa_exotic_aromatic}
\abs{\kappa}+\abs{\boldsymbol{\kappa}'}+1=2\abs{\gamma}.
\end{equation}
\end{lemma}

In the aromatic context, the order of an aromatic tree coincides with the number of nodes and~\eqref{equation:equality_kappa_exotic_aromatic} becomes
\begin{equation}
\label{equation:equality_kappa_aromatic}
\abs{\gamma}=\abs{\kappa}=\abs{\boldsymbol{\kappa}'}+1.
\end{equation}
If an exotic aromatic tree satisfies~\eqref{equation:equality_kappa_aromatic}, it does not imply that~$\gamma$ is aromatic. In fact, there exists an infinite number of exotic aromatic trees satisfying~\eqref{equation:equality_kappa_aromatic} that do not reduce to aromatic trees: the exotic aromatic trees with the same number of stolons and lianas.
In the context of branched rough paths, a similar identity to~\eqref{equation:equality_kappa_aromatic} on multi-indices is used in~\cite[eq.\ts (6.3)]{Linares21tsg}, and the composition~$\kappa$ is called the fertility in this context.

\begin{proof}[Proof of Lemma~\ref{lemma:def_order}]
The arrows are either part of a liana or are standard arrows whose source are nodes. We denote~$\Circled{A}_{\Circled{0}}^{\diamond}$ the latter set.
Similarly, the nodes of~$\gamma\in \Gamma_\kappa$ can be decomposed into two sets: the nodes that are the source of an arrow in~$\Circled{A}_{\Circled{0}}^{\diamond}$, gathered in~$V^{\diamond}$, and the ones that are the source of no arrows (the stolons). We observe that
\[\abs{V}=\abs{V^{\diamond}}+2\abs{S},\quad \abs{\Circled{A_0}}=\abs{\Circled{A}_{\Circled{0}}^{\diamond}}+2\abs{L},\quad \abs{V^{\diamond}}=\abs{\Circled{A}_{\Circled{0}}^{\diamond}}.\]
Each node in~$V^{\diamond}$ is the source of a unique arrow in~$\Circled{A}_{\Circled{0}}^{\diamond}$, so that~$\abs{V^{\diamond}}=\abs{\Circled{A}_{\Circled{0}}^{\diamond}}$.
Thus, we deduce
\[\abs{\kappa}+\abs{\boldsymbol{\kappa}'}+1=\abs{V}+\abs{\Circled{A_0}}
=2(\abs{V}+\abs{L}-\abs{S})
=2\abs{\gamma},\]
which gives the desired identity~\eqref{equation:equality_kappa_exotic_aromatic}.
\end{proof}

In Table~\ref{table:example_EAF}, we present the list of the exotic aromatic trees of order one and two (see also Appendix~\ref{section:list_trees_order_3} for the order three). On the contrary of the aromatic case, there exists an infinite number of exotic aromatic trees for a given number of nodes~$\abs{\kappa}>0$. Indeed, adding any number of lianas to an exotic aromatic tree does not modify the value of~$\abs{\kappa}$, but gives a different tree.
\begin{table}[!htb]
	\setcellgapes{2pt}
	\centering
	\begin{tabular}{|c|c|c|c|c|c|c|c|}
	\hline
	$\abs{\gamma}$ &~$\abs{\kappa}$ &~$\kappa$ &~$\boldsymbol{\kappa}'$ &~$\tau$ &~$\sigma$ &~$\gamma$ &~$\FF(\gamma)(f)$ \\
	\hhline{|=|=|=|=|=|=|=|=|}
	$1$ &~$1$ &~$(1)$ &~$(0)$ &  &~$(\Circled{0},1)$ &~$\eatree{1}{000}{1}$ &~$f^i\partial_i$ \\
	\hline
	$2$ &~$1$ &~$(0,0,1)$ &~$(0,0,2)$ &~$(1,1)$ &~$(\Circled{0},1)(\Circled{1},\Circled{2})$ &~$\eatree{001}{010}{1}$ &~$f_{jj}^i\partial_i$ \\
	\cline{6-8}
	 &  &  &  &  &~$(\Circled{0},\Circled{1})(\Circled{2},1)$ &~$\eatree{001}{110}{1}$ &~$f^j_{ij}\partial_i$ \\
	\hline
	$2$ &~$2$ &~$(1,1)$ &~$(0,1)$ &~$(1)$ &~$(\Circled{0},1)(\Circled{1},2)$ &~$\eatree{11}{000}{1}$ &~$f^i_j f^j \partial_i$ \\
	\cline{6-8}
	 &  &  &  &  &~$(\Circled{0},2)(\Circled{1},1)$ &~$\eatree{11}{100}{1}$ &~$f^j_j f^i \partial_i$ \\
	 \cline{6-8}
	 &  &  &  &  &~$(\Circled{0},\Circled{1})(1,2)$ &~$\eatree{11}{011}{1}$ &~$f^j f^j_i \partial_i$ \\
	\hline
	$2$ &~$3$ &~$(3)$ &~$(0)$ &  &~$(\Circled{0},1)(2,3)$ &~$\eatree{3}{001}{1}$ &~$f^i f^j f^j \partial_i$ \\
	\hline
	\end{tabular}
	\caption{List of the exotic aromatic trees of order one and two, with their associated composition, derived composition, target map, source map, and elementary differential (see Definition~\ref{definition:elementary_differential}). We use the notation~$\tau=(\tau(\Circled{1}),\dots,\tau(\Circled{\abs{\boldsymbol{\kappa}'}}))$.
}
	\label{table:example_EAF}
	\setcellgapes{1pt}
\end{table}

\subsection{Characterisation of exotic aromatic B-series}
\label{section:main_result}

Butcher trees represent elementary differentials and B-series represent Taylor expansions of the different flows used in numerical analysis~\cite{Hairer06gni}.
Each exotic aromatic tree represents an elementary differential given by the following definition, where we use the standard notation~$\partial_i$ for the vector basis of~$\R^d$.
\begin{definition}
\label{definition:elementary_differential}
Given a smooth vector field~$f\in\mathfrak{X}(\R^d)$ and an exotic aromatic tree~$\gamma=(V,\Circled{A_0},\sigma,\tau)\in \Gamma$, the elementary differential~$\FF_d(\gamma)$ associated to~$\gamma$ is the following vector field
%\footnote{The dependence in~$f$,~$d$, and~$x\in \R^d$ is often omitted when the context is clear.}
\[
\FF_d(\gamma)(f)=\sum_{i_1,\dots,i_{\abs{\kappa}}\in\{1,\dots,d\}}
\sum_{i_{\Circled{0}},\dots,i_{\Circled{\abs{\boldsymbol{\kappa}'}}}\in\{1,\dots,d\}}
\prod_{v\in V} f_{i_{\tau^{-1}(\{v\})}}^{i_{v}} \delta_{i_\sigma} \partial_{i_{\Circled{0}}},
\]
where~$i_{\tau^{-1}(\{v\})}=i_{l_1}\dots i_{l_m}$ for~$\tau^{-1}(\{v\})=\{l_1,\dots,l_m\}$ and~$\delta_{i_\sigma}=\prod_{j=1}^{\abs{\gamma}} \delta_{i_{p_j},i_{q_j}}$ for the source map~$\sigma=\prod_{j=1}^{\abs{\gamma}} (p_j, q_j)$ with~$\delta_{i,j}=1$ if~$i=j$ and 0 else.
The elementary differential map of an exotic aromatic tree~$\gamma$ is the following sequence of maps indexed by the dimension of the problem
\[\FF(\gamma)=(\FF_d(\gamma)\colon\mathfrak{X}(\R^d)\rightarrow \mathfrak{X}(\R^d))_d.\]
The maps~$\FF_d$ are extended by linearity to~$\Span(\Gamma)$.
\end{definition}

\begin{ex*}
The following exotic aromatic tree~$\gamma$ represents the elementary differential:
\[
\gamma=\eatree{201}{011}{2}, \quad
\FF(\gamma)(f)=\sum_{i_v,i_{\Circled{a}}} f^{i_{1}} f^{i_2} f_{i_{\Circled{1}}i_{\Circled{2}}}^{i_3} 
\delta_{i_{\Circled{0}},i_1} \delta_{i_{\Circled{1}},i_{\Circled{2}}} \delta_{i_2,i_3}
\partial_{i_{\Circled{0}}}
=\sum_{i,j,k} f^{i} f^{k} f_{jj}^{k} \partial_{i}
=(f,\Delta f)f.
\]
Further examples are presented in Table~\ref{table:example_EAF}. Note that in the elementary differentials, every index appears twice. For aromatic trees, every index appears both at the top and at the bottom, while this is not the case in general for exotic aromatic trees.
\end{ex*}

For a fixed dimension~$d$, the elementary differential map~$\FF_d$ is not injective in general. There can be multiple ways to write a given elementary differential with exotic aromatic trees if the dimension~$d$ is too low.
For instance, in dimension~$d=1$, all the trees with composition~$\kappa$ represent the same elementary differential
\[\FF_1(\gamma)(f)=\prod_{j=0}^\infty (f^{(j)})^{\kappa(j)}.\]
This is a strong motivation for considering sequences of maps~$\FF(\gamma)=(\FF_d(\gamma))_d$ indexed by the dimension of the problem.
We prove in Section~\ref{section:proof_strong_equiv} that exotic aromatic trees represent independent elementary differentials.

\begin{remark}
The elementary differential extends to multi-aromas~$\gamma=(V,\Circled{A},\sigma,\tau) \in \Gamma^0$ by
\[
\FF_d(\gamma)(f)=
\sum_{i_1,\dots,i_{\abs{\kappa}}\in\{1,\dots,d\}}
\sum_{i_{\Circled{1}},\dots,i_{\Circled{\abs{\boldsymbol{\kappa}'}}}\in\{1,\dots,d\}}
 \prod_{v\in V} f_{i_{\tau^{-1}(\{v\})}}^{i_{v}} \delta_{i_\sigma}.
\]
\end{remark}

%\begin{remark}
%If we assume in addition that~$\sigma(\Circled{0})\in N$, the definitions of the exotic aromatic trees and their associated differential are equivalent to the ones presented in~\cite{Laurent20eab, Laurent21ocf}.
%When~$\sigma(\Circled{0})\in A$, we obtain a ghost liana, and we draw it as an outgoing dashed arrow.
%\end{remark}

An exotic aromatic B-series is a formal series indexed over exotic aromatic trees. As we consider Taylor expansions and thus use the grading by the number of nodes, we consider series with a finite number of trees with~$m$ nodes for all~$m$. This assumption is not required in the numerical applications~\cite{Laurent20eab,Laurent21ocf} as the expansions are graded naturally by the order of the trees and not by the number of nodes.
\begin{definition}
\label{definition:EAB}
Given a coefficient map~$b\colon \Gamma\rightarrow\R$ that has finite support on~$\Gamma_m$ for~$m>0$, the associated exotic aromatic B-series in dimension~$d$ is the following formal series
\[B_d(b)=\sum_{m>0} \sum_{\gamma\in\Gamma_m} b(\gamma)\FF_d(\gamma).\]
An exotic aromatic B-series is a sequence~$B(b)=(B_d(b))_d$ indexed by the dimension.
If the support of~$b$ is included in the set of aromatic/exotic/stolonic trees,~$B(b)$ is called an aromatic/exotic/stolonic B-series.
\end{definition}

%When applied in numerical analysis, the term~$h^{\abs{\gamma}}$ is added in the sum in Definition~\ref{definition:EAB}, where~$h$ is the stepsize of the numerical scheme and~$\abs{\gamma}$ is the order of the tree.
%In the deterministic context,~$\abs{\gamma}$ and~$\abs{\kappa}$ coincides (where~$\kappa$ is the composition of~$\gamma$). In the stochastic context, it is not the case anymore as a scaling is added for the lianas, due to the terms of size~$h^{1/2}$.

The first main result of this paper is the characterization of exotic aromatic B-series with orthogonal-equivariance and locality.
\begin{theorem}
\label{theorem:orthogonal_equivariance_EAB}
Let~$\varphi=(\varphi_d\colon \mathfrak{X}(\R^d)\rightarrow \mathfrak{X}(\R^d))_d$ be a sequence of smooth maps. Then, the Taylor expansion of~$\varphi_d$ around the trivial vector field~$0$ is an exotic aromatic B-series~$\varphi_d=B_d(b_d)$ if and only if~$\varphi_d$ is local and orthogonal-equivariant.
If, in addition,~$\varphi$ is trivially decoupling, then there exists a coefficient map~$b\colon \Gamma\rightarrow\R$ such that~$\varphi=B(b)$.
\end{theorem}

Theorem~\ref{theorem:orthogonal_equivariance_EAB} provides a simple geometric criterion for checking whether a modified vector field corresponds to an exotic aromatic B-series.
In addition, it confirms that the exotic aromatic B-series are a natural extension of the aromatic B-series as they both satisfy similar universal geometric properties~\cite{MuntheKaas16abs}.

We are then interested in characterising the different subsets of exotic aromatic B-series and in particular the exotic B-series, as they play an important role in stochastic numerical analysis~\cite{Laurent20eab}.
%For instance in~\cite{Laurent20eab}, the analysis relies on exotic B-series only.
We give the following characterisation using categorical equivariance properties.
\begin{theorem}
\label{theorem:isometric_equivariance_EAB}
Let~$\varphi=(\varphi_d\colon \mathfrak{X}(\R^d)\rightarrow\mathfrak{X}(\R^d))_d$ be a local sequence of smooth maps.
Then, the Taylor expansion of~$\varphi$ around the trivial vector field~$0$ is:
\begin{itemize}
\item a stolonic B-series if and only if~$\varphi$ is left-orthogonal-equivariant,
\item an exotic B-series if and only if~$\varphi$ is right-orthogonal-equivariant,
\item a B-series if and only if~$\varphi$ is semi-orthogonal-equivariant,
\item a connected exotic aromatic B-series if and only if~$\varphi$ is orthogonal-equivariant and decoupling.
\end{itemize}
In particular, affine equivariance and semi-orthogonal-equivariance are equivalent as they both characterise B-series.
\end{theorem}

Theorem~\ref{theorem:isometric_equivariance_EAB} is a natural result in the sense that the larger the set of equivariant transformations, the smaller the associated class of Butcher series. For instance, as semi-orthogonal equivariance implies right-orthogonal equivariance, the exotic B-series contain standard B-series.

\begin{remark}
The decoupling property exactly corresponds to the connectedness of the graphs involved in the expansion.
This link was first observed in~\cite{McLachlan16bsm} in the context of aromatic B-series, where a decoupling aromatic B-series is showed to be a connected aromatic B-series, i.e., a standard B-series.
Observe also that the elementary differential of an exotic aromatic tree can be factored through its connected components: let~$\mu_1$, \dots,~$\mu_m\in \Gamma_c^0$,~$\tau\in \Gamma_c$ and the exotic aromatic tree~$\gamma=\mu_1\dots\mu_m \tau$, then
\begin{equation}
\label{equation:factorisation_connected_components}
F_d(\gamma)(f)(x)=F_d(\mu_1)(f)(x)\dots F_d(\mu_m)(f)(x)F_d(\tau)(f)(x).
\end{equation}
\end{remark}

%\begin{remark}
%A direct consequence of Theorem~\ref{theorem:isometric_equivariance_EAB} is that affine equivariance~\cite{MuntheKaas16abs} and semi-orthogonal-equivariance are equivalent concepts as they both characterise B-series.
%We give a straightforward proof of this equivalence.
%It is clear that affine equivariance implies semi-orthogonal-equivariance.
%Assume~$\varphi$ is semi-orthogonal-equivariant and let~$A\in \R^{d_2\times d_1}$,~$b\in \R^{d_2}$,~$f_1\in \mathfrak{X}(\R^{d_1})$,~$f_2\in \mathfrak{X}(\R^{d_2})$ such that~$f_2(Ax+b)=Af_1(x)$. Then, using the singular value decomposition, we can assume without loss of generality that~$A$ is diagonal. Moreover, as~$\varphi$ is trivially decoupling, we can also assume that~$d_1=d_2$ and that~$A$ is invertible.
%\modg{Thus, we are left to prove that semi-orthogonal-equivariance implies equivariance with respect to diagonal matrices.}
%\end{remark}

The classification of exotic aromatic B-series is summarised in Table~\ref{table:classification of B-series}.
To the best of our knowledge, the equivalence of affine-equivariance and semi-orthogonal-equivariance is a new non-trivial result.
We derive in Subsection~\ref{section:degeneracies_2} a simplified characterisation related to the numerical analysis literature under a regularity assumption on the vector fields.
\begin{table}[!htb]
	\setcellgapes{2pt}
	\centering
	\begin{tabular}{|c|c|}
	\hline
	geometric property & associated Butcher series \\
	\hhline{|=|=|}
	orthogonal-equivariance & exotic aromatic B-series \\
	\hline
	$\GL$-equivariance & aromatic B-series \\
	\hline
	left-orthogonal-equivariance & stolonic B-series \\
	\hline
	right-orthogonal-equivariance & exotic B-series \\
	\hline
	affine/semi-orthogonal-equivariance & B-series\\
	\hline
	\end{tabular}
	\caption{Classification of B-series with respect to their equivariance properties (see Theorems~\ref{theorem:orthogonal_equivariance_EAB} and~\ref{theorem:isometric_equivariance_EAB}).}
	\label{table:classification of B-series}
	\setcellgapes{1pt}
\end{table}

\section{Geometric characterisation of exotic aromatic B-series}
\label{section:proof}

This section is devoted to the proof of Theorem~\ref{theorem:orthogonal_equivariance_EAB}.
Following~\cite{MuntheKaas16abs}, we first restrict our study to symmetric multilinear local equivariant maps defined on the infinite jet bundle at one point in Section~\ref{section:decomposition_invariant_subspaces}. We then decompose our space into invariant tensor spaces using the invariant tensor theorem. In Section~\ref{section:surjection_Exotic}, we draw a one-to-one correspondence between tensors in the invariant spaces and exotic aromatic trees.
Section~\ref{section:conclusion_proof} contains the proof of Theorem~\ref{theorem:orthogonal_equivariance_EAB} and a clarifying example.

\subsection{Invariant tensor spaces and transfer of geometric properties}
\label{section:decomposition_invariant_subspaces}

Let~$\varphi_d\colon\mathfrak{X}(\R^d)\rightarrow\mathfrak{X}(\R^d)$ be a smooth local~$G$-equivariant map with~$G=H\ltimes \R^d$.
The Taylor expansion of~$\varphi_d$ around the vector field~$0$ is
\begin{equation}
\label{equation:Taylor_exp}
\sum_{m\geq 1} \frac{1}{m!} D^m\varphi_d(0)(f,\dots,f),
\end{equation}
where~$\varphi_d(0)=0$ by locality~\cite[Lem.\ts 6.1]{McLachlan16bsm}.
We denote~$M$ the tangent space of~$\R^d$ at the origin~$0$ (note that~$M\equiv \R^d$),~$T^m M=M\otimes \dots \otimes M$ the tensor product of~$m$ copies of~$M$, and~$S^m M$ the symmetrised tensor product.
Following the transfer argument~\cite[Thm.\ts 3.9]{MuntheKaas16abs}, the~$m$-th Taylor term~$D^m\varphi_d(0)$ inherits the locality and orthogonal-equivariance properties.
Moreover, the Peetre theorem~\cite[\S\ts 19.9]{Kolar93noi} and the equivariance~\cite[Thm.\ts 5.6]{MuntheKaas16abs} allow us to assume without loss of generality that the~$m$-th Taylor term is in the space of multilinear symmetric local~$H$-equivariant maps:
\[D^m\varphi_d(0)\in \LL_H(S^m(M\otimes SM^\ast),M),\]
where the action of~$H$ on~$S^m(M\otimes SM^\ast)$ is the natural action induced on tensor spaces, and for a vector space~$V$,~$S V:=\bigoplus_{j=0}^\infty S^j V$ is the symmetric algebra.
%~$M=T_0\R^d\equiv \R^d$ is the tangent space at the origin, the action of~$H$ on~$S^m(M\otimes SM^\ast)$ is the natural action induced on tensor spaces, and for a vector space~$V$,~$S V:=\bigoplus_{j=0}^\infty S^j V$ is the symmetric algebra.
This result works for any isotropy group~$H$.

Given~$\kappa\colon\N\rightarrow\N$, we define the tensor space~$\TT_\kappa$ and its symmetric counterpart~$\SS_\kappa$ by
\[
\TT_\kappa=M\otimes \bigotimes_{j=0}^\infty T^{\kappa(j)}(M^\ast\otimes T^j M), \quad
\SS_\kappa=M\otimes \bigotimes_{j=0}^\infty S^{\kappa(j)}(M^\ast\otimes S^j M),
\]
and their~$H$-invariant subspaces~$\TT_\kappa^H$ and~$\SS_\kappa^H$.
Then,~\cite[Thm.\ts 5.6]{MuntheKaas16abs} gives the isomorphism
\[
\LL_H(S^m(M\otimes SM^\ast),M)\equiv \bigoplus_{\abs{\kappa}=m}\SS_\kappa^H.
\]
%Our question is now reformulated as the following: if~$H=\Or_d(\R)$, can any tensor in~$\SS_\kappa^H$ be rewritten as a linear combination of exotic aromatic trees?
In the affine case~$H=\GL_d(\R)$, it is shown in~\cite[Thm.\ts 6.3]{MuntheKaas16abs} with the description of~$H$-invariant tensors~\cite[\S\ts 24.3]{Kolar93noi} that~$\LL_H(S^m(M\otimes SM^\ast),M)$ is a finite dimensional space. In the orthogonal case~$H=\Or_d(\R)$, this property does not hold in general, and we get the following instead.
\begin{theorem}
\label{theorem:scaling_property_O}
Let~$H=\Or_d(\R)$ and~$m$ a positive integer, the following isomorphism holds
\[
%\overset{\bullet}{\LL}_G(S^m\mathfrak{X}(\R^d),\mathfrak{X}(\R^d))\equiv \LL_H(S^m(M\otimes SM^\ast),M) \equiv \bigoplus_{\underset{\abs{\kappa}+\abs{\boldsymbol{\kappa}'}+1 \in 2\Z}{\abs{\kappa}=m}}\SS_\kappa^H.
\LL_H(S^m(M\otimes SM^\ast),M) \equiv \bigoplus_{\underset{\abs{\kappa}+\abs{\boldsymbol{\kappa}'}+1 \in 2\Z}{\abs{\kappa}=m}}\SS_\kappa^H.
\]
\end{theorem}

\begin{proof}
%We already know that
%\[\overset{\bullet}{\LL}_G(S^m\mathfrak{X}(\R^d),\mathfrak{X}(\R^d))\equiv\bigoplus_{\abs{\kappa}=m}\SS_\kappa^H.\]
Following the description of~$\Or_d(\R)$-invariant tensors~\cite[\S\ts 33.2]{Kolar93noi}, we deduce that~$\TT_\kappa^H$ is trivial when~$\abs{\kappa}+\abs{\boldsymbol{\kappa}'}+1$ is odd. As~$\SS_\kappa^H$ is naturally injected into~$\TT_\kappa^H$, we obtain the desired result.
\end{proof}

We write~$\Delta^m\varphi_d \colon\mathfrak{X}(\R^d)\rightarrow\mathfrak{X}(\R^d)$ the term of order~$m\geq 1$ in the expansion of~$\varphi_d$ around the vector field~$0$, that is,
\[\Delta^m\varphi_d(f)=D^m\varphi_d(0)(f,\dots,f).\]
%The term of order~$m=0$ vanishes by locality~\cite[Lem.\ts 6.1]{McLachlan16bsm}.
Thanks to Theorem~\ref{theorem:scaling_property_O},~$\Delta^m\varphi_d$ has the following form,
\begin{equation}
\label{equation:Taylor_expansion_varphi}
\Delta^m\varphi_d(f)=\sum_{\underset{\abs{\kappa}+\abs{\boldsymbol{\kappa}'}+1 \in 2\Z}{\abs{\kappa}=m}} \psi_{\kappa,d}(f^\kappa),\quad f^\kappa=(\underbrace{f,\dots,f}_{\kappa(0)},\underbrace{f',\dots,f'}_{\kappa(1)},\dots),
\end{equation}
where only a finite number of the~$\psi_{\kappa,d}\in \SS_\kappa^{\Or_d(\R)}$ are non-zero.
%The \moda{categorical equivariance} properties of~$\varphi$ are transferred automatically to the~$\Delta^m\varphi$ (see for instance~\cite[Prop.\ts 6.2]{McLachlan16bsm}).
%The decoupling and~$0$-preserving properties also transfer to the Taylor terms.

\begin{lemma}
Let~$\varphi$ be a local orthogonal-equivariant sequence of smooth maps.
The categorical equivariance, decoupling, or trivially decoupling properties of~$\varphi$ are transferred to the Taylor terms~$\Delta^m\varphi$ in~\eqref{equation:Taylor_expansion_varphi}.
\end{lemma}

\begin{proof}
The proof for the transfer of equivariance properties is the same as in~\cite[Prop.\ts 6.2]{McLachlan16bsm}.
Let us prove the transfer of the decoupling property. The transfer of the trivially decoupling property uses the same arguments.
The Taylor terms of~$\varphi$ satisfy~\cite[\S\ts 5.11]{Kriegl97tcs}
\[\Delta^m\varphi_d(f)=\partial_{t_1\dots t_m} \varphi_d((t_1+\dots_+t_m)f)\Big \vert_{t_1=\dots=t_m=0}.\]
Thus, we find
\begin{align*}
\Delta^m\varphi_d(f_1\oplus f_2)&=\partial_{t_1\dots t_m} \varphi_d((t_1+\dots_+t_m)f_1\oplus(t_1+\dots_+t_m)f_2)\Big \vert_{t_1=\dots=t_m=0}\\
&=\partial_{t_1\dots t_m} \Big[ \varphi_d((t_1+\dots_+t_m)f_1)\oplus \varphi_d((t_1+\dots_+t_m)f_2)\Big] \Big\vert_{t_1=\dots=t_m=0}\\
&=\Delta^m\varphi_d(f_1)\oplus \Delta^m\varphi_d(f_2).
\end{align*}
Hence the result.
\end{proof}

Thus, we restrict our study for the rest of the paper to the~$\Delta^m\varphi=(\Delta^m\varphi_d)_d$ that can be expressed as finite sums of tensors in~$\SS_\kappa^{\Or_d(\R)}$.
We show in Subsection~\ref{section:surjection_Exotic} that the expansion~\eqref{equation:Taylor_expansion_varphi} corresponds to the order~$m$ term of an exotic aromatic B-series.

%\begin{corollary}
%For~$m>0$,~$\LL_H(S^m(M\otimes SM^\ast),M)$ is an infinite dimensional space.
%\end{corollary}
%
%\begin{proof}
%Just add as many lianas as you want! It does not modify~$\kappa$, but it changes the tensor.
%\end{proof}

\subsection{Correspondence between exotic aromatic trees and invariant tensors}
\label{section:surjection_Exotic}

Let us now draw a correspondence between exotic aromatic trees and tensors in~$\SS_\kappa^{\Or_d(\R)}$.
\begin{theorem}
\label{theorem:surjection_EAF}
For a given~$\kappa$, there exists a surjective linear map~$\widetilde{\FF}_d\colon \Span(\Gamma_\kappa)\rightarrow \SS_\kappa^{\Or_d(\R)}$.
The map~$\widetilde{\FF}_d$ is a bijection if and only if~$2d\geq \abs{\kappa}+\abs{\boldsymbol{\kappa}'}+1$.
Moreover, the elementary differential map~$\FF_d$ is injective on~$\Span(\Gamma_\kappa)$ if~$2d\geq \abs{\kappa}+\abs{\boldsymbol{\kappa}'}+1$.
\end{theorem}

\begin{proof}
\textbf{Decomposition of~$\TT_\kappa$.}
Following Theorem~\ref{theorem:scaling_property_O}, we assume without loss of generality that~$\abs{\kappa}+\abs{\boldsymbol{\kappa}'}+1 =2d_0$ is even.
We rewrite~$\TT_\kappa$ as
\[
\TT_\kappa=M\otimes \bigotimes_{j=0}^\infty \bigotimes_{i=1}^{\kappa(j)} T_i^j, \quad T_i^j=M^\ast\otimes T^j M.
\]
We number the~$2d_0$ components of~$\TT_\kappa$ in the following way. The copies of~$M^\ast$ in~$\TT_\kappa$ are numbered in an arbitrary manner from~$1$ to~$\abs{\kappa}$ and the copies of~$M$ from~$\Circled{0}$ to~$\abs{\Circled{\boldsymbol{\kappa}'}}$ so that
\[
\TT_\kappa=M_{\Circled{0}}\otimes \bigotimes_{j=0}^\infty \bigotimes_{i=1}^{\kappa(j)} T_i^j.
\]
If the numbering is given by
\[
T_i^j=M^\ast_n\otimes M_{\Circled{n_1}}\otimes \cdots \otimes M_{\Circled{n_j}},
\]
then we write
\[\tau(\Circled{n_k})=n,\quad k=1,\dots,j.\]
This defines the target map~$\tau\colon \Circled{A} \rightarrow V$, the arrows~$\Circled{A}=\{\Circled{1},\dots,\Circled{\abs{\boldsymbol{\kappa}'}}\}$,~$\Circled{A_0}=\{\Circled{0}\}\cup \Circled{A}$, and the vertices~$V=\{1,\dots,\abs{\kappa}\}$.

%\textbf{Intermediate diagram.}
We now define the linear maps~$\omega$,~$\pi$, and~$\delta$ in order to obtain the following diagram.
\begin{center}
\begin{tikzcd}
\TT_\kappa^{\Or_d(\R)} \arrow[r, twoheadrightarrow, "\pi"] 
& \SS_\kappa^{\Or_d(\R)} \\
\Span(\Sigma_\kappa) \arrow[r, twoheadrightarrow, "\omega"] \arrow[u, twoheadrightarrow, "\delta"] 
& \Span( \Gamma_\kappa )  
\end{tikzcd}
\end{center}

\textbf{Definition of~$\omega$.}
We denote~$\Sigma_\kappa$ the set of permutations~$\sigma$ of the set~$V\cup \Circled{A_0}$ that have no fixed point and that satisfy~$\sigma\circ\sigma=\text{id}$.
Given~$\sigma\in \Sigma_\kappa$ and the target map~$\tau$, there exists a unique exotic aromatic tree~$(V,\Circled{A_0},\sigma,\tau)\in \Gamma_\kappa$ according to Definition~\ref{definition:exotic_aromatic_trees}.
%\begin{itemize}
%\item if~$\sigma(\Circled{0})\in N$, then the root is~$\sigma(\Circled{0})$ and it has a standard outgoing edge. If~$\sigma(\Circled{0})\in A$, the root is given by~$\tau(\sigma(\Circled{0}))$ and it has a outgoing liana.
%\item if~$n\in N$ and~$\Circled{A_0}=\sigma(n)\in A$, the arrow~$\Circled{A_0}$ links~$n$ to~$\tau(\Circled{A_0})$,
%\item if~$n_1\in N$ and~$n_2=\sigma(n_1)\in N$, a stolon links~$n_1$ and~$n_2$,
%\item if~$\Circled{a_1}\in A$ and~$\Circled{a_2}=\sigma(\Circled{a_1})\in A$, a liana links~$\tau(\Circled{a_1})$ and~$\tau(\Circled{a_2})$.
%\end{itemize}
%It is straightforward to show that we obtain an exotic aromatic tree of composition~$\kappa$ with this process.
This yields a map~$\omega\colon \Sigma_\kappa\rightarrow \Gamma_\kappa$. We extend this map by linearity to obtain~$\omega\colon \Span( \Sigma_\kappa) \rightarrow \Span( \Gamma_\kappa)$.

\textbf{Definition of~$\pi$.}
The projection map~$\pi\colon \TT_\kappa\rightarrow\SS_\kappa$ is compatible with the action of~$\Or_d(\R)$. Thus, it induces a surjective linear map (still denoted~$\pi$ for simplicity) from~$\TT_\kappa^{\Or_d(\R)}$ to~$\SS_\kappa^{\Or_d(\R)}$.

\textbf{Definition of~$\delta$.}
Using the isomorphism~$\TT_\kappa \equiv \LL(\bigotimes_{n\in V\cup \Circled{A_0}} M_n,\R)$, we define~$\delta(\sigma)$ for a permutation~$\sigma\in\Sigma_\kappa$ by
\[
\delta(\sigma)(v)=\prod_{\underset{j=\sigma(i), i<j}{i,j\in V\cup \Circled{A_0}}}(v_{i},v_j), \quad v=\bigotimes_{n\in V\cup \Circled{A_0}} v_n \in \bigotimes_{n\in V\cup \Circled{A_0}} M_n,
\]
where~$(.,.)$ is the standard scalar product in~$\R^d$ and where we fixed an arbitrary total order on~$V\cup \Circled{A_0}$.
We extend~$\delta$ by linearity on~$\Span( \Sigma_\kappa)$.
The surjectivity of~$\delta$ is a consequence of the~$\Or_d(\R)$-invariant tensor theorem~\cite[\S\ts 33.2]{Kolar93noi} (see also~\cite[Sec.\ts II.9]{Weyl39tcg} and~\cite[Sec.\ts 10.2]{Kraft96cit}).
Moreover,~$\delta$ is a bijection if and only if~$2d\geq \abs{\kappa}+\abs{\boldsymbol{\kappa}'}+1=\vert V\vert + \vert\Circled{A_0}\vert$ (see~\cite[Sec.\ts II.17]{Weyl39tcg}).
% also p31 statement, p33 end discussion and p35 solution/discussion basis
%Note that the treatment of the covariant and contravariant vectors is the same, on the opposite of the~$\GL_d(\R)$-invariant tensor theorem.

\textbf{Action of~$G_\kappa$.}
The target function~$\tau\colon \Circled{A}\rightarrow V$ is an element of~$V^{\Circled{A}}$. We denote~$\Sigma_{\Circled{A}}$ the set of permutations of the arrows in~$\Circled{A}$, respectively~$\Sigma_V$ the set of permutations of the nodes in~$V$, and~$\Sigma_{\Circled{A}} \times \Sigma_V$ the permutations of~$V\cup \Circled{A_0}$ that leave~$\Circled{0}$ fixed, permute the elements in~$V$, and the elements of~$\Circled{A}$ without mixing them.
The action of an element of~$g\in\Sigma_{\Circled{A}} \times \Sigma_V$ on~$\xi\in V^{\Circled{A}}$ is
\[g\cdot \xi=g\circ \xi \circ g^{-1}.\]
We denote~$G_\kappa$ the stabilizer of the target function~$\tau$, that is,
\[G_\kappa=\{g\in\Sigma_{\Circled{A}} \times \Sigma_V,\quad g\cdot \tau=\tau\}.\]
The permutations in~$G_\kappa$ represent the permutations of arrows and nodes that are compatible with the target map~$\tau$.

\textbf{Definition of~$K_\Sigma$.}
An element~$g\in \Sigma_{\Circled{A}} \times \Sigma_V$ acts naturally on~$\sigma\in\Sigma_\kappa$ by~$g\cdot \sigma=g\circ\sigma\circ g^{-1}$.
We observe that~$\omega(\sigma_1)=\omega(\sigma_2)$ if and only if there exists~$g\in G_\kappa$ such that~$\sigma_1=g\cdot \sigma_2$.
We define~$K_\Sigma$ as the vector subspace of~$\Span(\Sigma_\kappa)$ spanned by the~$g_1\cdot \sigma-g_2\cdot \sigma$ for~$g_1$,~$g_2\in G_\kappa$ and~$\sigma\in \Sigma_\kappa$. By definition,~$K_\Sigma$ is the kernel of~$\omega$, so that the following sequence is exact.
\begin{center}
\begin{tikzcd}
0 \arrow[r] & K_\Sigma \arrow[r, hookrightarrow] & \Span(\Sigma_\kappa) \arrow[r, twoheadrightarrow, "\omega"] & \Span( \Gamma_\kappa) \arrow[r] & 0
\end{tikzcd}
\end{center}

\textbf{Definition of~$K_\otimes$.}
Using the identification~$\TT_\kappa \equiv \LL(\bigotimes_{n\in V\cup \Circled{A_0}} M_n,\R)$, the action of an element of~$\Sigma_{\Circled{A}} \times \Sigma_V$ on~$\TT_\kappa$ is
\[(g\cdot \varphi)(v)=\varphi(\bigotimes_{n\in V\cup \Circled{A_0}} v_{g(n)}), \quad \varphi\in \LL(\bigotimes_{n\in V\cup \Circled{A_0}} M_n,\R), \quad v=\bigotimes_{n\in V\cup \Circled{A_0}} v_n.\]
We observe that by definition of~$\tau$,~$\pi(\varphi_1)=\pi(\varphi_2)$ if and only if~$\varphi_1=g\cdot \varphi_2$ with~$g\in G_\kappa$.
We define~$K_\otimes$ as the vector space spanned by the~$g_1\cdot\varphi-g_2\cdot\varphi$ for~$g_1$,~$g_2\in G_\kappa$,~$\varphi\in \TT_\kappa$.
By definition,~$K_\otimes$ is the kernel of~$\pi$, and is also the kernel of the restriction~$\pi\colon\TT_\kappa^{\Or_d(\R)}\rightarrow \SS_\kappa^{\Or_d(\R)}$.
We have the following exact sequence.
\begin{center}
\begin{tikzcd}
0 \arrow[r] & K_\otimes \arrow[r, hookrightarrow] & \TT_\kappa^{\Or_d(\R)} \arrow[r, twoheadrightarrow, "\pi"] & \SS_\kappa^{\Or_d(\R)} \arrow[r] & 0
\end{tikzcd}
\end{center}

\textbf{Definition of~$\widetilde{\FF}_d$.}
The action of~$G_\kappa$ commutes with~$\delta$, that is,~$\delta(g\cdot \sigma)=g\cdot\delta( \sigma)$.
%Thus~$\delta$ induces a surjective linear map from~$K_\Sigma$ to~$K_\otimes$.
Thus,~$\delta$ induces a linear map from~$K_\Sigma$ to~$K_\otimes$. By the fundamental theorem on homomorphisms, there exists a surjective map~$\widetilde{\delta}$ from~$\Span( \Sigma_\kappa)/K_\Sigma$ to~$\TT_\kappa^{\Or_d(\R)}/K_\otimes$.
We obtain the following diagram, where~$\widetilde{\FF}_d=\pi\circ\widetilde{\delta}\circ\omega^{-1}$ is surjective.
\begin{center}
\begin{tikzcd}
\TT_\kappa^{\Or_d(\R)}/K_\otimes \arrow[r, hook, two heads, "\pi"] 
& \SS_\kappa^{\Or_d(\R)} \\
\Span( \Sigma_\kappa)/K_\Sigma \arrow[r, hook, two heads, "\omega"] \arrow[u, twoheadrightarrow, "\widetilde{\delta}"] 
& \Span( \Gamma_\kappa) \arrow[u, "\widetilde{\FF}_d", two heads, dashed] 
\end{tikzcd}
\end{center}
The map~$\widetilde{\FF}_d$ is bijective if and only if~$\widetilde{\delta}$ is bijective, that is, if and only if~$2d\geq \abs{\kappa}+\abs{\boldsymbol{\kappa}'}+1$.
The elementary differential map~$\FF_d$ in Definition~\ref{definition:elementary_differential} is related to the map~$\widetilde{\FF}_d$ on~$\Span(\Gamma_\kappa)$ by
\[\FF_d(\gamma)(f)(x)=\widetilde{\FF}_d(\gamma)(f^\kappa),\quad f^\kappa(x)=(\underbrace{f(x),\dots,f(x)}_{\kappa(0)},\underbrace{f'(x),\dots,f'(x)}_{\kappa(1)},\dots), \quad f\in \mathfrak{X}(\R^d).\]
Thus, if~$\widetilde{\FF}_d$ is bijective on~$\Span(\Gamma_\kappa)$, then~$\FF_d$ is injective on~$\Span(\Gamma_\kappa)$.
\end{proof}

\begin{remark}
There exists a finite number of exotic aromatic trees of composition~$\kappa$, so that~$\SS_\kappa^{\Or_d(\R)}$ is finite-dimensional.
However, on the contrary of the~$\GL$-equivariance setting, there is an infinite number of exotic aromatic trees with a given number of nodes~$\abs{\kappa}$, so that~$\LL_{\Or_d(\R)}(S^m(M\otimes SM^\ast),M)$ is infinite-dimensional.
%Indeed, adding any number of lianas to an exotic aromatic tree does not modify its composition order~$\abs{\kappa}$.
\end{remark}

%\begin{remark}
%If~$\sigma\in\Sigma_\kappa$ is such that~$\sigma(N)=A$, then~$\omega(\sigma)$ is an aromatic tree.
%\end{remark}

\subsection{Characterisation of exotic aromatic B-series}
\label{section:conclusion_proof}

As first mentioned in~\cite[Sec.\ts 3.2.4]{Laurent21ata}, the exotic aromatic B-series satisfy geometric properties.
\begin{proposition}
\label{proposition:equivariance_ppty}
The exotic aromatic B-series are local, orthogonal-equivariant, and trivially decoupling.
\end{proposition}

\begin{proof}
The locality and trivially decoupling properties are straightforward from Definition~\ref{definition:elementary_differential}.
Let~$g=(A,b)\in \Or_d(\R)\ltimes \R^d$ and~$\gamma\in \Gamma$, then
\begin{align*}
\FF_d(\gamma)(g\cdot f)(x)
&=\sum_{\underset{v\in V,\Circled{a}\in \Circled{A_0}}{i_v, i_{\Circled{a}}}} \prod_{v\in V} \sum_{\underset{v\in V,\Circled{a_0}\in \Circled{A}}{k_v, k_{\Circled{a_0}}}} a_{i_v,k_v} a_{i_{\tau^{-1}(\{v\})},k_{\tau^{-1}(\{v\})}} f_{k_{\tau^{-1}(\{v\})}}^{k_{v}} (A^{-1}x-b) \delta_{i_\sigma} \partial_{i_{\Circled{0}}}\\
&=\sum_{\underset{v\in V,\Circled{a}\in \Circled{A}}{i_{\Circled{0}},k_v, k_{\Circled{a}}}} a_{i_{\Circled{0}},k_{\sigma(\Circled{0})}} \prod_{v\in V} f_{k_{\tau^{-1}(\{v\})}}^{k_{v}} (A^{-1}x-b) \delta_{k_\sigma} \partial_{i_{\Circled{0}}}\\
&=(g\cdot \FF_d(\gamma)(f))(x),
\end{align*}
where~$a_{i_J,k_J}=\prod_{j\in J} a_{i_j,k_j}$ and we used that~$A^T A=I_d$.
By linearity, exotic aromatic B-series are orthogonal-equivariant.
\end{proof}

The trivially decoupling property characterises the sequences of elementary differentials associated to exotic aromatic trees independently of the dimension~$d$.
\begin{proposition}
\label{proposition:isomorphism_F}
Let~$\varphi=(\FF_d(\gamma_d))_d$ be trivially decoupling, with~$\gamma_d\in \Span(\Gamma_\kappa)$, then there exists a unique~$\gamma\in \Span(\Gamma_\kappa)$ such that~$\varphi=\FF(\gamma)$.
\end{proposition}

\begin{proof}
%Let~$\varphi=(\varphi_d)_d\in (S_\kappa^{\Or(d)})_d$.
%Using Theorem~\ref{theorem:surjection_EAF}, for all~$d$, there exists~$\gamma_d\in \Span(\Gamma_\kappa)$ such that~$\FF_d(\gamma_d)=\varphi_d$.
Let~$d_1\leq d_2$, then we have~$\FF_{d_2}(\gamma_{d_2})(f_1\oplus 0)=\FF_{d_1}(\gamma_{d_1})(f_1)\oplus 0$ for~$f_1\in \mathfrak{X}(\R^{d_1})$ as~$\varphi$ is trivially decoupling.
On the other hand, a close inspection of Definition~\ref{definition:elementary_differential} yields that~$\FF_{d_2}(\gamma_{d_2})(f_1\oplus 0)=\FF_{d_1}(\gamma_{d_2})(f_1)\oplus 0$, and we deduce
\begin{equation}
\label{equation:proof_F_surjectivity_dimensions}
\FF_{d_1}(\gamma_{d_1})=\FF_{d_1}(\gamma_{d_2}), \quad d_1\leq d_2.
\end{equation}
Let~$d_0=(\abs{\kappa}+\abs{\boldsymbol{\kappa}'}+1)/2$. For~$d\geq d_0$, equation~\eqref{equation:proof_F_surjectivity_dimensions} gives~$\FF_d(\gamma_d)=\FF_d(\gamma_{d_0})$.
For~$d\leq d_0$, equation~\eqref{equation:proof_F_surjectivity_dimensions} gives~$\FF_{d_0}(\gamma_{d_0})=\FF_{d_0}(\gamma_d)$. As~$\FF_{d_0}$ is injective on~$\Span(\Gamma_\kappa)$,~$\gamma_d=\gamma_{d_0}$.
Thus, we obtain~$\varphi=\FF(\gamma_{d_0})$.
The uniqueness of~$\gamma_{d_0}$ is a consequence of the injectivity of~$\FF_{d_0}$ on~$\Span(\Gamma_\kappa)$ (see Theorem~\ref{theorem:surjection_EAF}).
%Let~$\gamma\in \Span(\Gamma)$ be such that~$\FF(\gamma)=0$ and let~$m$ be the maximum order of the exotic aromatic trees in~$\gamma$.
%As~$\FF_{m}(\gamma)=0$, Theorem~\ref{theorem:surjection_EAF} yields~$\gamma=0$. Thus,~$\FF$ is injective.
\end{proof}

Let us now prove the characterisation of exotic aromatic B-series.
\begin{proof}[Proof of Theorem~\ref{theorem:orthogonal_equivariance_EAB}]
Let~$\varphi_d\colon \mathfrak{X}(\R^d)\rightarrow \mathfrak{X}(\R^d)$ be a local and orthogonal-equivariant map.
Thanks to Theorem~\ref{theorem:scaling_property_O}, the term of order~$m$ in the Taylor expansion~\eqref{equation:Taylor_exp} of~$\varphi_d$ around the~$0$ vector field has the form~\eqref{equation:Taylor_expansion_varphi}.
%\begin{equation}
%\label{equation:Taylor_expansion_varphi}
%\sum_{m\geq 1} \sum_{\underset{\abs{\kappa}+\abs{\boldsymbol{\kappa}'}+1 \in 2\Z}{\abs{\kappa}=m}} \psi_{\kappa,d}(\underbrace{f(x),\dots,f(x)}_{\kappa(0)},\underbrace{f'(x),\dots,f'(x)}_{\kappa(1)},\dots),\quad \psi_{\kappa,d}\in \SS_\kappa^{\Or_d(\R)}.
%\end{equation}
Theorem~\ref{theorem:surjection_EAF} gives the existence of~$\gamma_{\kappa,d}\in \Span(\Gamma_\kappa)$ such that~$\psi_{\kappa,d}=\widetilde{\FF}_d(\gamma_{\kappa,d})$.
The Taylor expansion of~$\varphi_d$ around the~$0$ vector field thus is the exotic aromatic B-series:
\[\sum_{m\geq 1} \frac{1}{m!} \sum_{\underset{\abs{\kappa}+\abs{\boldsymbol{\kappa}'}+1 \in 2\Z}{\abs{\kappa}=m}} \FF_d(\gamma_{\kappa,d})(f)(x).\]
If, in addition,~$\varphi$ is trivially decoupling, Proposition~\ref{proposition:isomorphism_F} gives the existence of the coefficient map~$b\colon \Gamma\rightarrow\R$ such that~$\varphi=B(b)$.
%\modg{The uniqueness of~$b$ is a direct consequence of Proposition~\ref{proposition:dual_vf}.}
\end{proof}

\begin{ex*}
Let us illustrate the tensor spaces and maps from the proof of Theorem~\ref{theorem:surjection_EAF} for the composition~$\kappa=(2,0,1)$ (in the spirit of the examples in~\cite{Markl08git,MuntheKaas16abs}). The tensor space has the form
\[
\TT_\kappa=M_{\Circled{0}}\otimes M_1^\ast \otimes M_2^\ast \otimes (M_3^\ast \otimes M_{\Circled{1}} \otimes M_{\Circled{2}}).
\]
The associated set of nodes and arrows are~$V=\{1,2,3\}$ and~$\Circled{A_0}=\{\Circled{0},\Circled{1},\Circled{2}\}$. The target map~$\tau\colon \Circled{A}\rightarrow V$ is given by~$\tau(\Circled{1})=\tau(\Circled{2})=3$. The stabilizer of~$\tau$ is
\[G_\kappa=\{\id,(1,2),(\Circled{1},\Circled{2}),(1,2)(\Circled{1},\Circled{2})\}.\]
We present the output of~$\omega$,~$\widetilde{\delta}$ and~$\widetilde{\FF}$ for the different~$\sigma\in\Sigma_\kappa$ in Table~\ref{table:example_values_proof_surjection_EAF}, where we gather together the~$G_\kappa$-orbits. We write~$\widetilde{\delta}(\sigma)$ as an element of 
\[
\TT_\kappa/K_\otimes\equiv\LL(T^2 M\otimes\LL(T^2 M, M),M)/K_\otimes,
\]
that is, for~$v$,~$w\in M$ and a bilinear map~$\zeta\in \LL(T^2 M, M)$, we have~$\widetilde{\delta}(\sigma)(v,w,\zeta)\in M$.
For~$\widetilde{\FF}$, we use the identification
\[
\SS_\kappa\equiv\LL(S^2 M\otimes\LL(S^2 M, M),M).
\]
Replacing~$v=w=f(x)$ and~$\zeta=f''(x)$ yields the elementary differential~$\FF(\gamma)(f)(x)$ of Definition~\ref{definition:elementary_differential}.
Note that the first two lines of Table~\ref{table:example_values_proof_surjection_EAF} are aromatic trees, and also appear in~\cite[Table 2]{MuntheKaas16abs}. It can be seen directly on the associated permutations~$\sigma$, as each arrow is paired with a node and vice versa.
%The first four trees fit in the formalism of exotic aromatic trees as presented in~\cite{Laurent20eab,Laurent21ocf}. The last two trees are new as they have a ghost liana.

\begin{table}[htb]
	\setcellgapes{2pt}
	\centering
	\begin{tabular}{|c|c|c|c|}
	\hline
	$\sigma\in\Sigma_\kappa$ &~$\gamma=\omega(\sigma)$ &~$\widetilde{\delta}(\sigma)(v,w,\zeta)$ &~$\widetilde{\FF}(\gamma)(v,w,\zeta)$ \\
	\hhline{|=|=|=|=|}
	$(\Circled{0},3)(\Circled{1},1)(\Circled{2},2)$ & \multirow{2}{*}{$\eatree{201}{000}{1}$} &~$\zeta^i(v,w) \partial_i$ &  \\
	$(\Circled{0},3)(\Circled{1},2)(\Circled{2},1)$ &  &~$\zeta^i(w,v)\partial_i$ &~$\zeta^i(v,w)\partial_i$ \\
	\hline
	$(\Circled{0},1)(\Circled{1},2)(\Circled{2},3)$ &  &~$\zeta^j(w,\partial_j)v^i\partial_i$ &  \\
	$(\Circled{0},1)(\Circled{1},3)(\Circled{2},2)$ &  &~$\zeta^j(\partial_j,w)v^i\partial_i$ &  \\
	$(\Circled{0},2)(\Circled{1},1)(\Circled{2},3)$ & \multirow{2}{*}{$\eatree{201}{100}{1}$} &~$\zeta^j(v,\partial_j)w^i\partial_i$ &  \\
	$(\Circled{0},2)(\Circled{1},3)(\Circled{2},1)$ &  &~$\zeta^j(\partial_j,v)w^i\partial_i$ &~$\frac{1}{2}(\zeta^j(w,\partial_j)v^i+\zeta^j(v,\partial_j)w^i)\partial_i$ \\
	\hline
	$(\Circled{0},3)(\Circled{1},\Circled{2})(1,2)$ &~$\eatree{201}{011}{1}$ &~$(v,w)\zeta^i(\partial_j,\partial_j)\partial_i$ &~$(v,w)\zeta^i(\partial_j,\partial_j)\partial_i$ \\
	\hline
	$(\Circled{0},1)(\Circled{1},\Circled{2})(2,3)$ & \multirow{2}{*}{$\eatree{201}{011}{2}$} &~$(w,\zeta(\partial_j,\partial_j))v^i\partial_i$ &  \\
	$(\Circled{0},2)(\Circled{1},\Circled{2})(1,3)$ &  &~$(v,\zeta(\partial_j,\partial_j))w^i\partial_i$ &~$\frac{1}{2}((w,\zeta(\partial_j,\partial_j))v^i+(v,\zeta(\partial_j,\partial_j))w^i)\partial_i$ \\
	\hline
	$(\Circled{0},\Circled{1})(\Circled{2},1)(2,3)$ &  &~$(w,\zeta(\partial_i,v))\partial_i$ &  \\
	$(\Circled{0},\Circled{1})(\Circled{2},2)(1,3)$ &  &~$(v,\zeta(\partial_i,w))\partial_i$ &  \\
	$(\Circled{0},\Circled{2})(\Circled{1},1)(2,3)$ & \multirow{2}{*}{$\eatree{201}{011}{3}$} &~$(w,\zeta(v,\partial_i))\partial_i$ &  \\
	$(\Circled{0},\Circled{2})(\Circled{1},2)(1,3)$ &  &~$(v,\zeta(w,\partial_i))\partial_i$ &~$\frac{1}{2}((w,\zeta(\partial_i,v))+(v,\zeta(\partial_i,w)))\partial_i$ \\
	\hline
	$(\Circled{0},\Circled{1})(\Circled{2},3)(1,2)$ & \multirow{2}{*}{$\eatree{201}{111}{1}$} &~$(v,w)\zeta^j(\partial_i,\partial_j)\partial_i$ &  \\
	$(\Circled{0},\Circled{2})(\Circled{1},3)(1,2)$ &  &~$(v,w)\zeta^j(\partial_j,\partial_i)\partial_i$ &~$(v,w)\zeta^j(\partial_j,\partial_i)\partial_i$ \\
	\hline
	\end{tabular}
	\caption{Outputs of the functions~$\omega$,~$\delta$ and~$\widetilde{\FF}$ appearing in the proof of Theorem~\ref{theorem:surjection_EAF} for the composition~$\kappa=(2,0,1)$ and target map~$\tau(\Circled{1})=\tau(\Circled{2})=3$. The bilinear map~$\zeta$ is assumed symmetric in the last column. The sums on all involved indices are omitted for simplicity.}
	\label{table:example_values_proof_surjection_EAF}
	\setcellgapes{1pt}
\end{table}

\end{ex*}

\section{Classification of exotic aromatic B-series}
\label{section:proof_strong_equiv}

This section is devoted to the proof of the finer classification of Theorem~\ref{theorem:isometric_equivariance_EAB}. The proof, presented in Subsection~\ref{section:strong_characterisations}, relies heavily on the use of new dual vector fields, that we introduce in Subsection~\ref{section:degeneracies}. We present the impact of degeneracies on the classification in Subsection~\ref{section:degeneracies_2}.

\subsection{Dual vector fields}
\label{section:degeneracies}

%Given a molecule or a connected exotic aromatic tree~$\gamma \in \Gamma^0_c\cup\Gamma^1_c$, define its symmetry coefficient~$\upsigma(\gamma)$ as the number of automorphisms of the vertices and arrows of~$\gamma$ that preserve the graph structure.

The exotic aromatic trees given in Definition~\ref{definition:exotic_aromatic_trees} produce independent elementary differentials. To the best of our knowledge, this is a new non-trivial result, that plays an important role in the proof of Theorem~\ref{theorem:isometric_equivariance_EAB}.
The standard approach for proving such a property is to consider dual vector fields, in the spirit of the works~\cite{Hairer06gni, Iserles07bsm, McLachlan16bsm, Laurent20eab}.
%In the proof of Theorem~\ref{theorem:isometric_equivariance_EAB}, we use the following dual vector fields.
%\cite[Chap.\ts 3, Ex.\ts 3]{Hairer06gni},~\cite[Sec.\ts 4]{Iserles07bsm},~\cite[Sec.\ts 4.2]{McLachlan16bsm}, and~\cite[Rk.\ts 4.8]{Laurent20eab}.
\begin{proposition}
\label{proposition:dual_vf}
Given an exotic aromatic tree or multi-aroma~$\gamma\in \Gamma^0\cup \Gamma$, index the coordinates of~$\R^{\abs{\gamma}}$ by the uplets in~$(V\cup \Circled{A_0})/\sigma$ (respectively in~$(V\cup \Circled{A})/\sigma$ if~$\gamma\in \Gamma^0$), where~$v$ and~$\sigma(v)$ are identified. This corresponds to the nodes~$(v,\Circled{a})\in V^{\diamond}$ that are not part of stolons, the stolons~$s=(v_1,v_2)\in S$, and the lianas~$l=(\Circled{a_1},\Circled{a_2})\in L$.
%We then split~$V\cup A$ into the nodes/edges that are standard nodes, part of stolons, and part of lianas. We number each of these three sets starting from 1.
Let~$\theta^\gamma$ be the following parameter indexed by the standard nodes, the nodes in stolons, and the arrows in lianas,
%indexed by~$V^{\diamond}=\{1,\dots,\abs{\V^{\diamond}}\}$ and the arrows in~$\Circled{A}=\{\Circled{1},\dots,\Circled{\abs{\boldsymbol{\kappa}'}}\}$:
%(where~$A_L$ stands for the arrows that are part of a liana, not including~$\Circled{0}$,~$V_S$ the nodes that are part of a stolon, and~$\theta_{l_r}$ is included only if~$\gamma$ has a ghost liana)
%\[\theta=(\theta_{l_r},\theta_{l_1,1},\theta_{l_1,2},\theta_{l_2,1}\dots,\theta_{s_1,1},\theta_{s_1,2},\dots)^T \in \R^{\abs{A_L}+\abs{V_S}}.\]
%\[\theta^\gamma=(\theta_{1},\dots,\theta_{\abs{V}},\theta_{\Circled{1}},\dots)^T \in \R^{2\abs{\gamma}-1}.\]
\[\theta^\gamma=(\theta^{V^{\diamond}},\theta^{S},\theta^{L})
=(\theta_{1}^{V^{\diamond}},\dots,\theta_{\abs{V^{\diamond}}}^{V^{\diamond}},
\theta_{1}^{S},\dots,\theta_{2\abs{S}}^{S},
\theta_{\Circled{1}}^L,\dots,\theta_{\Circled{2\abs{L}}}^{L}).\]
Define the associated vector field~$f_\gamma^{(\theta^\gamma)}\in \mathfrak{X}(\R^{\abs{\gamma}})$ by
%Label the nodes of~$\gamma$ that are not part of a stolon~$v_i$ with~$i=1,\dots, \abs{N}-2\abs{S}$, the stolons~$s_i$ with~$i=\abs{N}-2\abs{S}+1,\dots,\abs{N}-\abs{S}$, and the lianas~$l_i$ with~$i=\abs{N}-\abs{S}+1,\dots,\abs{\gamma}$.
\begin{align*}
f_\gamma^{(\theta^\gamma),v}(x)&= \theta_{v}^{V^{\diamond}}\prod_{\tau(\Circled{a})=v} \theta_{\Circled{a}}^L x_{\Circled{a}},\\
f_\gamma^{(\theta^\gamma),s}(x)&= \theta_{v_1}^S \prod_{\tau(\Circled{a})=v_1} \theta_{\Circled{a}}^L x_{\Circled{a}}+\theta_{v_2}^S \prod_{\tau(\Circled{a})=v_2} \theta_{\Circled{a}}^L x_{\Circled{a}}, \quad s=(v_1,v_2)\in S,\\
f_\gamma^{(\theta^\gamma),l}(x)&=0,
\end{align*}
where an empty product is 1 and~$\theta_{\Circled{a}}^L=1$ if~$\Circled{a}\notin L$. By convention, if~$\gamma$ has a root, the coordinate of the root is the first one.
Let~$\gamma$,~$\hat{\gamma}\in \Gamma^0\cup \Gamma$, then 
\[
(\FF_{\abs{\hat{\gamma}}}(\gamma)(f_{\hat{\gamma}}^{(\theta^{\hat{\gamma}})}))^1_{\theta^\gamma}\Big\vert_{\theta=0}(0)= 0 \text{ if } \hat{\gamma}\neq \mu \gamma, \quad \mu\in \Gamma_0.
\]
In particular, the elementary differential map~$\FF$ is injective on~$\Span(\Gamma)$.
Moreover, for connected graphs~$\gamma$,~$\hat{\gamma}\in \Gamma^0_c\cup\Gamma_c$, we find
\[
(\FF_{\abs{\hat{\gamma}}}(\gamma)(f_{\hat{\gamma}}^{(\theta^{\hat{\gamma}})}))^1_{\theta^\gamma}\Big\vert_{\theta=0}(0)=\upsigma(\gamma)\neq 0 \text{ if and only if } \gamma= \hat{\gamma},
\]
%where~$f^{S}_{\gamma}=\prod_{s\in S} f^{s}_{\gamma}$.
where the constant~$\upsigma(\gamma)$ is the symmetry coefficient of~$\gamma$, that is, the number of bijections of the vertices and arrows of~$\gamma$ that preserve the graph structure.
\end{proposition}

\begin{remark}
\label{remark:def_general_dual_vf}
Given an exotic aromatic tree~$\gamma=\mu_m \dots \mu_1 \tau$, we enforce an order on the aromas, so that~$\mu_1 \mu_2$ is now different from~$\mu_2 \mu_1$ if~$\mu_1\neq \mu_2$.
Consider now the additional parameter~$\theta^{\gamma}=(\theta^\tau,\theta^{\mu_1},\dots,\theta^{\mu_m})$, where the numbering of the nodes, lianas and stolons starts with~$\tau$, and continues in order with the~$\mu_i$.
With this order on the aromas and the numbering of~$\theta$, the first statement of Proposition~\ref{proposition:dual_vf} is then replaced by
\[
(\FF_{\abs{\hat{\gamma}}}(\gamma)(f_{\hat{\gamma}}^{(\theta^{\hat{\gamma}})}))^1_{\theta^\gamma}\Big\vert_{\theta=0}(0)\neq 0 \text{ if and only if } \hat{\gamma}= \mu \gamma, \quad \mu\in \Gamma_0.
\]
\end{remark}

\begin{proof}
Let~$\gamma=(V,\Circled{A},\sigma,\tau)$,~$\hat{\gamma}=(\hat V,\Circled{\hat{A}},\hat \sigma,\hat \tau)\in \Gamma^0\cup\Gamma$.
Definition~\ref{definition:elementary_differential} rewrites as
\begin{align*}
%\label{equation:detail_dual_vf_sum}
\FF_{\abs{\hat{\gamma}}}(\gamma)(f_{\hat{\gamma}}^{(\theta^{\hat{\gamma}})})(x)
&=\sum_{i\colon (V\cup \Circled{A})/\sigma \rightarrow (\hat{V}\cup \Circled{\hat{A}})/\hat{\sigma}} \prod_{v\in V} (f_{\hat{\gamma}})_{i_{\tau^{-1}(\{v\})}}^{i_{v}} \partial_{i_r}\\
&=\sum_{i\colon (V\cup \Circled{A})/\sigma \rightarrow (\hat{V}\cup \Circled{\hat{A}})/\hat{\sigma}}
\prod_{v\in V} \Big(
\theta_{i_v}^{V^{\diamond}} \theta_{\hat{\tau}^{-1}(i_v)}^L x_{\hat{\tau}^{-1}(i_v)}\ind_{i_v\in \hat{V}_0}\\&
+(\theta_{\hat{v}_1}^S \theta_{\hat{\tau}^{-1}(\hat{v}_1)}^L x_{\hat{\tau}^{-1}(\hat{v}_1)}+\theta_{\hat{v}_2}^S \theta_{\hat{\tau}^{-1}(\hat{v}_2)}^L x_{\hat{\tau}^{-1}(\hat{v}_2)})\ind_{i_v=(\hat{v}_1,\hat{v}_2)\in \hat{S}}
\Big)_{i_{\tau^{-1}(\{v\})}} \partial_{i_r}.
% \bigg(\prod_{\underset{i_v\in \hat{V}_0}{\Circled{\hat{a}}\in \hat{\tau}^{-1}(i_v)}} x_{\Circled{\hat{a}}}\bigg)_{i_{\tau^{-1}(\{v\})}}
%\bigg(\prod_{\underset{i_v\in \hat{S}}{\Circled{\hat{a}}\in \hat{\tau}^{-1}(i_v)}} x_{\Circled{\hat{a}}}\bigg)_{i_{\tau^{-1}(\{v\})}},
\end{align*}
where we fix~$\partial_{i_r}=\partial_{1}$ if~$\gamma\in \Gamma^0$.
By definition of the map~$i$, if~$\sigma(x)=y$,~$\hat{\sigma}(i_{x})=i_y$. Moreover, it is necessary that~$\hat{\tau}^{-1}(i_v)= i_{\tau^{-1}(\{v\})}$ for~$v\in V^{\diamond}$ (and analogously for~$v\in S$) so that~$(f_{\hat{\gamma}}^{(\theta^{\hat{\gamma}})})_{i_{\tau^{-1}(\{v\})}}^{i_{v}}(0)\neq 0$. Thus, the map~$i$ is compatible with the source and target maps. In particular,~$i$ sends predecessors of~$v$ to predecessors of~$i_v$.

On the other hand, the~$\theta$ parameter enforces the injectivity of~$i$ and it forces~$i$ to send stolons to stolons, lianas to lianas, nodes in~$V^{\diamond}$ to nodes in~$\hat{V}^{\diamond}$.
Thus~$i$ sends~$\gamma$ to a subgraph of~$\hat{\gamma}$. If~$\hat{\gamma}\neq\mu \gamma$, then at least an edge is missing and~$\FF_{\abs{\hat{\gamma}}}(\gamma)(f_{\hat{\gamma}}^{(\theta^{\hat{\gamma}})})(x)$ is a non-constant polynomial in~$x$, so that it vanishes at~$x=0$.

If~$\gamma$,~$\hat{\gamma}\in \Gamma^0_c\cup\Gamma_c$, the only maps~$i$ such that~$(\FF_{\abs{\hat{\gamma}}}(\gamma)(f_{\hat{\gamma}}^{(\theta^{\hat{\gamma}})}))^1_{\theta^\gamma}\Big\vert_{\theta=0}(0)\neq 0$ are the graph isomorphisms between~$\gamma$ and~$\hat{\gamma}$.
The number of such maps~$i$ is~$\upsigma(\gamma)$.
\end{proof}

\begin{ex*}
Consider the following exotic aromatic tree with its associated vector field and elementary differential
\[
\gamma=\eatree{011}{021}{1},
\quad
f_\gamma^{(\theta^\gamma)}\begin{pmatrix}
x_1\\
x_2\\
x_3
\end{pmatrix}=\begin{pmatrix}
0\\
\theta_{1}^S \theta_{\Circled{2}}^L \theta_{\Circled{3}}^L x_{3}^2+\theta_{2}^S \theta_{\Circled{1}}^L x_{1}\\
0
\end{pmatrix},
\quad (\FF_3(\gamma)(f_\gamma^{(\theta^\gamma)}))^1_{\theta^{\gamma}}\Big\vert_{\theta=0}(0)=2,
\]
where the coordinates represent in descending order the root, the stolon, and the liana.
Consider now the aroma
\[
\gamma=\earoma{02}{011}{1}, \quad
f_\gamma^{(\theta^\gamma)}\begin{pmatrix}
x_1\\
x_2
\end{pmatrix}=\begin{pmatrix}
\theta_{1}^{S} \theta_{\Circled{1}}^L x_{2}+\theta_{2}^S \theta_{\Circled{2}}^L x_{2}\\
0
\end{pmatrix}, \quad
\theta^{\gamma}=(\theta_{1}^{S},\theta_{2}^{S},\theta_{\Circled{1}}^L,\theta_{\Circled{2}}^L),
\]
and the tree
\[
\hat{\gamma}=\eatree{11}{000}{1}, \quad
f_{\hat{\gamma}}^{(\theta^{\hat{\gamma}})}\begin{pmatrix}
x_1\\
x_2
\end{pmatrix}=\begin{pmatrix}
\theta_{1}^{V^{\diamond}} x_{2}\\
\theta_{2}^{V^{\diamond}}
\end{pmatrix}, \quad
\theta^{\hat{\gamma}}=(\theta_{1}^{V^{\diamond}},\theta_{2}^{V^{\diamond}}).
\]
A calculation yields
\[
(\FF_{2}(\gamma)(f_{\hat{\gamma}}^{(\theta^{\hat{\gamma}})}))^1(x)=(\theta_{1}^{V^{\diamond}})^2, \quad (\FF_{2}(\gamma)(f_{\hat{\gamma}}^{(\theta^{\hat{\gamma}})}))^1_{\theta^\gamma}\Big\vert_{\theta=0}(0)=0.
\]
Note that fixing~$\theta=\ind$ yields~$(\FF_{2}(\gamma)(f_{\hat{\gamma}}^{(\ind)}))^1(0)=1$, so that the dual vector field without the~$\theta$ parameter fails to identify the difference between~$\gamma$ and~$\hat{\gamma}$.
The main use of the~$\theta$ parameter in the proof of Proposition~\ref{proposition:dual_vf} is to enforce the map~$i$ to be injective and to preserve the nature of each pair~$(v,\sigma(v))$. The dual vector field without the~$\theta$ parameter is not sufficient, even in the aromatic context:
\[\gamma=\eatree{02}{100}{1}, \quad\hat{\gamma}=\eatree{01}{100}{1}, \quad \FF_{\abs{\hat{\gamma}}}(\gamma)(f_{\hat{\gamma}}^{(\ind)})(0)=1.\]
This reveals a typographical error in~\cite[Rk.\ts 4.8]{Laurent20eab} where the remark only applies to exotic trees, and a minor error in~\cite[Sec.\ts 4.2]{McLachlan16bsm}. The further proofs of this paper can be adapted straightforwardly to fix the proofs in~\cite{McLachlan16bsm}.
\end{ex*}

\subsection{Categorical characterisations}
\label{section:strong_characterisations}

This section is devoted to the proof of Theorem~\ref{theorem:isometric_equivariance_EAB}.
Following Subsection~\ref{section:decomposition_invariant_subspaces}, as the regularity assumptions of Theorem~\ref{theorem:isometric_equivariance_EAB} imply the locality, orthogonal-equivariance, and trivially decoupling properties, we work directly with~$\varphi=\FF(\gamma)$ and~$\gamma\in \Span(\Gamma)$.

\begin{proposition}
\label{proposition:strong_equivariance_ppties}
Connected exotic aromatic B-series are decoupling.
stolonic B-series are left-orthogonal-equivariant and exotic B-series are right-orthogonal-equivariant.
\end{proposition}

\begin{proof}
The decoupling property is straightforward from Definition~\ref{definition:elementary_differential}.
%Let~$A\in \Or_d(\R)$ (with~$b=0$ for simplicity) and~$\gamma\in \Gamma$, then
%\begin{align*}
%\FF_d(\gamma)(Af(A^{-1}.))(x)
%&=\sum_{\underset{v\in V,\Circled{A_0}\in A}{i_v, i_{\Circled{A_0}}}} \prod_{v\in V} \sum_{\underset{v\in V,\Circled{A_0}\in \Circled{A}}{k_v, k_{\Circled{A_0}}}} a_{i_v,k_v} a_{i_{\tau^{-1}(\{v\})},k_{\tau^{-1}(\{v\})}} f_{k_{\tau^{-1}(\{v\})}}^{k_{v}} (A^{-1}x) \delta_{i_\sigma} \partial_{i_{\Circled{0}}}\\
%&=\sum_{\underset{v\in V,\Circled{A_0}\in \Circled{A}}{i_{\Circled{0}},k_v, k_{\Circled{A_0}}}} a_{i_{\Circled{0}},k_{\sigma(\Circled{0})}} \prod_{v\in V} f_{k_{\tau^{-1}(\{v\})}}^{k_{v}} (A^{-1}x) \prod_{\underset{p_j,q_j\neq \Circled{0}}{j=1}}^{\abs{\gamma}} \delta_{k_{p_j},k_{q_j}} \partial_{i_{\Circled{0}}}\\
%&=A\FF_d(\gamma)(f)(A^{-1}x),
%\end{align*}
%where~$a_{i_J,k_J}=\prod_{j\in J} a_{i_j,k_j}$ and we used that~$A^T A=I_d$.
%By linearity, exotic aromatic B-series~$B_d(b)$ are orthogonal-equivariant.
Let~$\gamma$ be an exotic aromatic tree without lianas and loops,~$f\in\mathfrak{X}(\R^{d_1})$,~$\hat{f}\in\mathfrak{X}(\R^{d_2})$,~$a(x)=Ax+b\in\leftiso_{d_1,d_2}$ with~$d_1 \leq d_2$.
Differentiating the identity~$\hat{f}(a(x))=Af(x)$ gives~$a_{K,J}\hat{f}_{K}^i(a(x))=a_{i,k} f_{J}^k(x)$.
We call leaves the vertices in~$V$ that are not the target of any arrow. We say a node~$v$ has depth~$p$ if the shortest path of~$v$ to a leaf passes through~$p$ different nodes (not including the start and end points). The nodes of depth at most~$p$ are gathered in the set~$V^{(p)}$. As~$\gamma$ does not have lianas or loops,~$V^{(\abs{\gamma})}=V$.
An induction on the depth yields
\begin{align*}
\FF_{d_2}(\gamma)(\hat{f})(a(x))
&=\sum_{i,k} \prod_{v\in V^{(0)}} a_{i_{v},k_{v}} f^{k_{v}} (x) \prod_{v\notin V^{(0)}} \hat{f}_{i_{\tau^{-1}(\{v\})}}^{i_{v}} (a(x)) \delta_{i_\sigma} \partial_{i_{\Circled{0}}}\\
&=\sum_{i,k} \prod_{v\in V^{(1)}} a_{i_{v},k_{v}}^{(1)} a_{k_{\tau^{-1}(\{v\})},i_{\tau^{-1}(\{v\})}}^{(1)} f^{k_{v}}_{k_{\tau^{-1}(\{v\})}}(x) \prod_{v\notin V^{(1)}} \hat{f}_{i_{\tau^{-1}(\{v\})}}^{i_{v}} (a(x)) \delta_{i_\sigma}^{(1)} \delta_{k_\sigma}^{(1)} \partial_{i_{\Circled{0}}}\\
&=\dots=\sum_{i,k} \prod_{v\in V^{(\abs{\gamma})}} a_{i_{v},k_{v}}^{(\abs{\gamma})} a_{k_{\tau^{-1}(\{v\})},i_{\tau^{-1}(\{v\})}}^{(\abs{\gamma})} f^{k_{v}}_{k_{\tau^{-1}(\{v\})}}(x) \delta_{i_\sigma}^{(\abs{\gamma})} \delta_{k_\sigma}^{(\abs{\gamma})} \partial_{i_{\Circled{0}}}\\
&=\sum_{i,k} \prod_{v\in V} a_{i_{\Circled{0}},k_{\sigma(\Circled{0})}} f^{k_{v}}_{k_{\tau^{-1}(\{v\})}}(x) \delta_{k_\sigma} \partial_{i_{\Circled{0}}}
=A\FF_{d_1}(\gamma)(f)(x),
\end{align*}
where~$a_{i_v,k_v}^{(p)}=a_{i_v,k_v}$ if~$\sigma(v)\in \tau^{-1}(w)$ with~$w\notin V^{(p)}$ and~$a_{i_v,k_v}^{(p)}=1$ else,~$\delta_{i_\sigma}^{(p)}$,~$\delta_{k_\sigma}^{(p)}$ contain the indices involved in the expression at step~$p$, and the sums are on all involved indices.

On the other hand, let the exotic tree~$\gamma$, the vector fields~$f\in\mathfrak{X}(\R^{d_1})$,~$\hat{f}\in\mathfrak{X}(\R^{d_2})$, and the affine transformation~$a(x)=Ax+b\in\rightiso_{d_1,d_2}$ with~$d_1 \geq d_2$. We find
\begin{align*}
\FF_{d_2}(\gamma)(g)(a(x))
&=\sum_{i,k} \prod_{v\in V} \hat{f}_{i_{\tau^{-1}(\{v\})}}^{i_{v}} (a(x)) \delta_{i_\sigma} \partial_{i_{\Circled{0}}}\\
&=\sum_{i,k} \prod_{v\in V} a_{i_{\tau^{-1}(\{v\})\cap L},k_{\tau^{-1}(\{v\})\cap L}} \hat{f}_{i_{\tau^{-1}(\{v\})}}^{i_{v}} (a(x)) \delta_{i_\sigma} \delta_{k_\sigma}^{(L)} \partial_{i_{\Circled{0}}},
\end{align*}
where we used that~$AA^T=I_{d_2}$ to add the coefficients associated to lianas and~$\delta_{k_\sigma}^{(L)}$ identifies the coefficients~$k$ associated to lianas.
The rest of the calculation is analogous to the left-orthogonal-equivariance case: we define the depth function on the tree without the lianas and we perform the calculation with an induction on the depth of the tree.
\end{proof}

\begin{proposition}
\label{proposition:partition_preserving_details}
Assume~$\varphi=\FF(\gamma)$ is decoupling, then~$\gamma\in \Span(\Gamma_c)$.
\end{proposition}

\begin{proof}
Let~$\gamma\in \Span(\Gamma)$ and~$\hat{\gamma}\in \Gamma$ be one of its exotic aromatic trees of maximal order among the ones that have at least one aroma, so that
%The combination of exotic aromatic trees~$\gamma$ decomposes into
\begin{equation}
\label{equation:decomposition_connected_components}
\gamma=c\hat{\gamma}+R, \quad \hat{\gamma}=\mu_m\dots\mu_1\tau, \quad c\in \R, \quad R\in \Span(\Gamma).
\end{equation}
Without loss of generality, we assume~$c=1$.
% is one of the aromas of maximal order in~$\gamma$, and~$p_1$ is the maximal power it is found. Then~$\mu_2$ is one of the aromas of maximal order in the terms of the form~$\mu_1^{p_1}\hat{\gamma}$ in~$\gamma$, and~$p_2$ is its maximal order, etc...
Define~$f_{\hat{\gamma}}^{(\theta^{\hat{\gamma}})}$ as in Proposition~\ref{proposition:dual_vf} and following the numbering of Remark~\ref{remark:def_general_dual_vf},
%For~$\lambda\in \R$, define the associated vector field (see Proposition~\ref{proposition:dual_vf})
\[
%\label{equation:def_concatenated_dual_vf}
f_{\hat{\gamma}}^{(\theta^{\hat{\gamma}})}=f_\tau^{(\theta^\tau)} \oplus f_{\mu_1}^{(\theta^{\mu_1})} \oplus \dots \oplus f_{\mu_m}^{(\theta^{\mu_m})} \in \mathfrak{X}(\R^d),
\]
where the first entry of~$f_{\hat{\gamma}}^{(\theta^{\hat{\gamma}})}$ corresponds to the root of~$\tau$.
Using~\eqref{equation:factorisation_connected_components} and Proposition~\ref{proposition:dual_vf}, we obtain
\[(\FF_d(\gamma)(f_{\hat{\gamma}}^{(\theta^{\hat{\gamma}})}))^1_{\theta^{\hat{\gamma}}}\Big\vert_{\theta=0}(0)=(\FF_d(\hat{\gamma})(f_{\hat{\gamma}}^{(\theta^{\hat{\gamma}})}))^1_{\theta^{\hat{\gamma}}}\Big\vert_{\theta=0}(0)\neq 0.\]
%where~$\theta^p=\theta_{\mu_1}^{p_1}\dots\theta_{\mu_m}^{p_m}\theta_\tau$ and~$P$ is a polynomial of degree~$p_1\abs{\mu_1}$.
%P(\FF_d(\mu_1)(f^{(\theta)}),\dots,\FF_d(\mu_m)(f^{(\theta)}))\FF_d(\tau)(f^{(\theta)})+\FF_d(\widehat{\gamma})(f^{(\theta)}),\[
%where, thanks to Proposition~\ref{proposition:dual_vf}, the elementary differentials satisfy for polynomials~$Q_j$,
%\]\partial_{\theta_j}\FF_d(\mu_j)(f^{(\theta)})\Big\vert_{\theta=0}(0)=\upsigma(\mu_j)\lambda_j^{\abs{\mu_j}}+ Q_j(\lambda_1,\dots,\lambda_{j-1}),\quad
%\partial_{\theta_j}(\FF_d(\hat{\gamma})(f^{(\theta)}))\Big\vert_{\theta=0}(0)=0.\[
%\modg{MEEEEH il va falloir faire des theta cooooompletement independants!!! meme pour chaque copie.}
%A calculation gives
%\begin{align*}
%(\FF_d(\gamma)(f))^1(0)
%%&= P(\upsigma(\mu_1)\lambda_1^{\abs{\mu_1}} f^{S_{\mu_1}}_{\mu_1}(x)+ Q_1(x,\lambda_2,\dots,\lambda_{m-1}),\dots,\\&\upsigma(\mu_{m-1})\lambda_{m-1}^{\abs{\mu_{m-1}}} f^{S_{\mu_{m-1}}}_{\mu_{m-1}}(x)+Q_{m-1}(x,\lambda_1,\dots,\lambda_{m-2}),\\&\upsigma(\mu_m)\lambda_m^{\abs{\mu_m}} f^{S_{\mu_m}}_{\mu_m}(x)+Q_{m}(x,\lambda_1,\dots,\lambda_{m-1}))\\
%&=\upsigma(\tau)P(\upsigma(\mu_1)\lambda_1^{\abs{\mu_1}},\dots,\upsigma(\mu_m)\lambda_m^{\abs{\mu_m}})
%+\widetilde{Q}(\lambda_1,\dots,\lambda_{m-1}).
%\end{align*}
On the other hand, as~$\FF(\gamma)$ is decoupling, we find
\[(\FF_d(\gamma)(f_{\hat{\gamma}}^{(\theta^{\hat{\gamma}})}))^1_{\theta^{\hat{\gamma}}}\Big\vert_{\theta=0}(0)=(\FF_{\abs{\tau}}(\gamma)(f_\tau^{(\theta^\tau)}))^1_{\theta^{\hat{\gamma}}}\Big\vert_{\theta=0}(0)=0,\]
as~$\theta^{\hat{\gamma}}$ holds more parameters than~$\theta^\tau$.
We obtain a contradiction and~$\gamma\in \Span(\Gamma_c)$.
\end{proof}

We now prove our second main result.
\begin{proof}[Proof of Theorem~\ref{theorem:isometric_equivariance_EAB}]
Assume~$\varphi=\FF(\gamma)$ is decoupling, then Proposition~\ref{proposition:partition_preserving_details} yields the connectedness of~$\gamma$.
Assume in addition that~$\varphi$ is local and right-orthogonal-equivariant and that at least one of the connected exotic aromatic trees~$\tau$ in~$\gamma$ has a stolon or a loop.
Consider~$d_1=\abs{\tau}$, index the coordinates of~$\R^{d_1}$ by~$(V\cup \Circled{A_0})/\sigma$ as in Proposition~\ref{proposition:dual_vf}.
We split the coordinates of~$\R^{d_1}$ into~$x=y\oplus z$, where~$y\in \R^{d_2}$ contains the coordinates that are not stolons or nodes in a loop and~$z$ the others. Let~$A\in\R^{d_2\times d_1}$ be the projection matrix on~$\R^{d_2}$, that is,~$A(y\oplus z)=y$.
Define~$f_1 =f_{\tau}^{(\theta^\tau)}$ and~$f_2(Ax)=Af_1(x)$. The right-orthogonal-equivariance property and Proposition~\ref{proposition:dual_vf} yield
%(which cannot be a stolon)
\[
(\varphi_{d_2}(f_2))^1_{\theta^\tau}\Big\vert_{\theta=0}(0)
=(\varphi_{\abs{\tau}}(f_\tau^{(\theta^\tau)}))^1_{\theta^\tau}\Big\vert_{\theta=0}(0)
=(F_{\abs{\tau}}(\tau)(f_\tau^{(\theta^\tau)}))^1_{\theta^\tau}\Big\vert_{\theta=0}(0)
=\upsigma(\tau)\neq 0
.
\]
As~$d_1>d_2$, there is at least one~$\theta_v^{V^{\diamond}}$ or~$\theta_v^{S}$ that does not appear in~$f_2$, but appears in~$\theta^\tau$.
Thus~$(\varphi_{d_2}(f_2))^1_{\theta^\tau}\Big\vert_{\theta=0}(0)=0$, which brings a contradiction.
%As~$f_2$ is the dual vector field of a subgraph of~$\tau$ that does not coincide with any exotic aromatic tree, an analogous reasoning to the one in Proposition~\ref{proposition:dual_vf} yields~$(\varphi_{d_2}(f_2))^1_{\theta}\Big\vert_{\theta=0}(0)=0$, which brings a contradiction.

Assume now that~$\varphi=F(\gamma)$ is left-orthogonal-equivariant and that at least one of the exotic aromatic trees in~$\gamma$ has a liana or a loop.
Consider the decomposition~\eqref{equation:decomposition_connected_components} of~$\gamma$ where~$\hat{\gamma}$ is the term of maximal order among the ones that have a liana or a loop.
For~$d_2=\abs{\hat{\gamma}}$, we split the coordinates of~$\R^{d_2}$ into~$x=y\oplus z$ (or~$x=z\oplus y$ if the root is a ghost liana, so that the root is still in first position), where~$y\in \R^{d_1}$ contains the coordinates that are not lianas or nodes in a loop and~$z$ the others. Let~$A\in\R^{d_2\times d_1}$ be such that~$A^T$ is the projection matrix on~$\R^{d_1}$, that is,~$A^T(y\oplus z)=y$. Define~$f_2=f_{\hat{\gamma}}^{(\theta^{\hat{\gamma}})}$ and~$f_1(y)=A^T f_2(y\oplus 0)$.
Proposition~\ref{proposition:dual_vf} and the left-orthogonal-equivariance give
\begin{align*}
(\varphi_{d_1}(f_1))^1_{\theta^{\hat{\gamma}}}\Big\vert_{\theta=0}(0)=(\varphi_{d_2}(f_2))^1_{\theta^{\hat{\gamma}}}\Big\vert_{\theta=0}(0)\neq 0.
\end{align*}
%and according to the left-orthogonal-equivariance, it should be equal to
%\begin{align*}
%\frac{(\varphi_{d_1}(f_1))^1(x)}{(Q(f^{S_{\mu_1}}_{\mu_1},\dots,f^{S_{\mu_m}}_{\mu_m}) f^{S_\tau}_\tau )(x\oplus 0)}.
%\end{align*}
As~$d_1<d_2$, there is at least a node or a liana of~$\hat{\gamma}$ that does not appear in~$f_1$, but appears in~$\theta^{\hat{\gamma}}$. We deduce~$(\varphi_{d_1}(f_1))^1_{\theta^{\hat{\gamma}}}\Big\vert_{\theta=0}(0)=0$, which brings a contradiction.
%As~$f_1$ is a polynomial vector field that does not correspond to any exotic aromatic tree, an analogous reasoning to the one in Proposition~\ref{proposition:dual_vf} yields~$(\varphi_{d_1}(f_1))^1_{\theta^p}\Big\vert_{\theta=0}(0)=0$, which brings a contradiction.

If~$\varphi$ is semi-orthogonal-equivariant, then~$\gamma$ is a linear combination of connected exotic aromatic trees without lianas, loops, and stolons, that is, a combination of standard Butcher trees.
\end{proof}

\subsection{Impact of degeneracies on the classification}
\label{section:degeneracies_2}

In a variety of contexts, the vector field~$f$ satisfies additional regularity properties.
For instance, if~$f$ is a polynomial map of order~$p$, then all exotic aromatic trees where at least a node is the target of more than~$p$ arrows have a trivial elementary differential. We mention in particular the work~\cite{Bogfjellmo22uat} on aromatic trees for quadratic differential equations that relies on such degeneracies.
In the original numerical application of the exotic aromatic B-series in molecular dynamics~\cite{Laurent20eab,Laurent21ocf} (see also~\cite{Lelievre10fec}), the vector fields~$f$ of interest are gradients, that is,~$f=\nabla V$ for a smooth function~$V\colon\R^d\rightarrow \R$.
In~\cite{Lelievre13onr,Duncan16vru,Abdulle19act}, vector fields of the form~$f=J\nabla V$ with~$J$ the standard symplectic matrix are used as perturbations to reduce the variance and accelerate the speed of convergence to equilibrium in the numerical integration of Langevin dynamics.
In this section, we update the classification of Theorem~\ref{theorem:isometric_equivariance_EAB} for gradient vector fields~$f=\nabla V$, gathered in the set~$\mathfrak{X}^\nabla(\R^d)$.
As discussed in~\cite[Remark 4.8]{Laurent20eab}, the gradient property of~$f$ translates into degeneracies.
%In particular, the exotic aromatic trees with ghost arrows are sufficient to represent all the elementary differentials in the gradient case, as done in~\cite{Laurent20eab,Laurent21ocf}.
\begin{proposition}[\cite{Laurent20eab,Laurent21ata}]
\label{proposition:degeneracies_gradient}
We say that two exotic aromatic trees~$\gamma_1$ and~$\gamma_2$ are equivalent on~$\mathfrak{X}^\nabla(\R^d)$, written~$\gamma_1\sim\gamma_2$, if by performing a finite number of the following operations, it is possible to transform~$\gamma_1$ into~$\gamma_2$:
\begin{itemize}
\item inversion edge-liana:~$\, \includegraphics[scale=0.55]{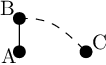} \, \sim \, \includegraphics[scale=0.55]{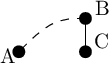} \,$ and~$\, \includegraphics[scale=0.55]{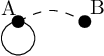} \, \sim \, \includegraphics[scale=0.55]{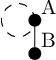} \,$,
\item inversion edge-stolon:~$\, \includegraphics[scale=0.55]{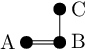} \, \sim \, \includegraphics[scale=0.55]{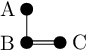} \,$,
\item simplification stolon-liana:~$\, \includegraphics[scale=0.55]{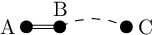} \, \sim \, \includegraphics[scale=0.55]{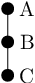} \,$ and~$\, \includegraphics[scale=0.55]{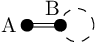} \, \sim \, \includegraphics[scale=0.55]{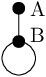} \,$.
\end{itemize}
The equivalence relation~$\sim$ preserves the composition~$\kappa$ of the graph and two equivalent exotic aromatic trees represent the same elementary differential~$\FF_d(\gamma_1)=\FF_d(\gamma_2)$ on~$\mathfrak{X}^\nabla(\R^d)$.
Moreover, for any connected exotic aromatic tree, there exists a unique exotic tree in its equivalence class.
\end{proposition}

%Proposition~\ref{proposition:degeneracies_gradient} is an extension of~\cite[Prop.\ts 4.7]{Laurent20eab}.
\begin{proof}
This is a direct consequence of the Schwarz theorem~$f^i_{j_1 \dots j_q}=f^{j_p}_{j_1 \dots j_{p-1} i j_{p+1} \dots j_q}$.
\end{proof}

\begin{ex*}
The following connected exotic aromatic trees are equivalent to exotic trees:
\[\eatree{201}{011}{3} \sim \eatree{201}{000}{1}, \quad
\eatree{12}{011}{4}\sim\eatree{12}{011}{2}\sim\eatree{12}{000}{1}, \quad 
\eatree{1001}{021}{1}\sim\eatree{1001}{110}{2}\sim\eatree{1001}{010}{1}.\]
%\, \includegraphics[scale=0.55]{Trees/simplificationexample1.eps} \, \sim \, \includegraphics[scale=0.55]{Trees/simplificationexample2.eps} \, \sim \, \includegraphics[scale=0.55]{Trees/simplificationexample3.eps} \, .\]
For further examples, the list of exotic aromatic trees of order 3 presented in Section~\ref{section:list_trees_order_3} gathers the equivalent exotic aromatic trees in adjacent lines.
\end{ex*}

On~$\mathfrak{X}^\nabla(\R^d)$, Theorem~\ref{theorem:orthogonal_equivariance_EAB} and Theorem~\ref{theorem:isometric_equivariance_EAB} reduce into the following result, which exactly characterises the exotic trees used in numerical analysis~\cite{Laurent20eab}.
\begin{theorem}
Let~$\varphi=(\varphi_d\colon \mathfrak{X}^\nabla(\R^d) \rightarrow\mathfrak{X}(\R^d))_d$ be a sequence of smooth maps.
%Then, the Taylor expansion of~$\varphi_d$ around the trivial vector field~$0$ is an exotic aromatic B-series with ghost arrow if and only if~$\varphi_d$ is local and orthogonal-equivariant.
The Taylor expansion of~$\varphi$ around the trivial vector field~$0$ is an exotic B-series on~$\mathfrak{X}^\nabla(\R^d)$ if and only if~$\varphi$ is local, orthogonal-equivariant, and decoupling.
\end{theorem}

\section{Conclusion and future works}
\label{section:future_work}

In this work, we proved that smooth local orthogonal-equivariant maps and exotic aromatic B-series represent the same object. This universal property shows that exotic aromatic B-series are not just a tool for calculations in numerical analysis, but a natural far-reaching algebraic object.
The analysis relies on the invariant tensor theorem for orthogonal-equivariant tensors and the Peetre theorem, but also on a new generalised construction of exotic aromatic trees.
In addition, we classified the intermediate subsets of exotic aromatic B-series, and in particular the exotic B-series, with respect to categorical equivariance properties. We also showed that exotic aromatic trees represent independent elementary differentials and identified the effect of the degeneracies appearing in numerical analysis on the classification.

A variety of theoretical and applied questions arise from the present work.
There exists a handful extensions of B-series used in numerical analysis such as partitioned B-series, exponential B-series \cite{Luan13ebs} or Lie-Butcher series \cite{Iserles00lgm}, and a variety of equivariance properties in~$\R^d$ but also on manifolds. We mention in particular the equivariance with respect to symplectic transformations. It would be interesting to draw geometric links between the different equivariance properties and the different types of B-series.
This could allow the creation of new extensions of B-series and to then find corresponding applications in numerical analysis.
For the B-series presented in this paper for instance, the stolonic B-series could be used in the study of projection methods and generalisations of the exotic aromatic formalism could be applied to the study of stochastic differential equations with multiplicative noise or to the creation of stochastic intrinsic methods of high order on manifolds (in the spirit of \cite{Iserles00lgm, MuntheKaas14ssa, Bharath23sae, MuntheKaas23lat, MuntheKaas24gio}).
%There are a variety of natural Lie subgroups~$G$ of~$\Diff(\R^d)$, where the~$G$-equivariant tensors are not known (symplectic group, ...).
%Moreover, a natural question would be to study the equivariance in~$\Diff(\MM)$ on a smooth manifold~$\MM$, and to find the associated B-series.
This is matter for future work.

%Question: what about diagonal equivariance, complex Lie groups, complex transf., symplectic group and partitioned B-series, on a manifold. Lie-B series?
%Envoyer questions à Olivier
%%
%TO DO
%
%\modg{Strong right-orthogonal-equivariance? Can we describe the exotic trees? Link with S-series and why we exactly have the exotic aromatic S-series...}
%\begin{itemize}
%\item Criterion to get Euclidean exotic aromatic B-series or EAB without lianas (to detect if we are in~$\R^d$ or on a manifold)
%\item Rewrite EAB with gradient conditions? Count? How does this work algebraically?
%\item B-series for different groups? We make a link between an invariant tensor theorem for a group~$H$ and a set of trees.
%\item What happens if we use sequences as in~\cite{McLachlan16bsm}?
%\item To every group~$H$ where we have an invariant tensor theorem, we can associate a set of trees. Other examples with different groups ??? Applications ?
%\item stochastic RK methods? Criterion? Application? Why do we not use the ghost liana?
%\item representation of EAF with only normal edges and arrows on them?
%\item Can the regularity of~$f$ be translated directly on the permutations, on~$G_\kappa$ ? What happens ? For the gradient property, we can exchange a liana with a node in the permutation.
%\item What happens with~$\phi$ the test function? Can we obtain a similar property with~$\phi$?
%\end{itemize}

%--------------------------------------------------------------
\bigskip

\noindent \textbf{Acknowledgements.}\
The work of Adrien Laurent was supported by the Research Council of Norway through project 302831 “Computational Dynamics and Stochastics on Manifolds” (CODYSMA).
%The author would like to thank XXX for helpful discussions.
%This work was partially supported by XXX.
%The computations were performed at the University of Geneva on the Baobab cluster using the Julia programming language.

\bibliographystyle{abbrv}
%\nocite{*}
\bibliography{Ma_Bibliographie}

\begin{thebibliography}{10}

\bibitem{Abdulle19act}
A.~Abdulle, G.~A. Pavliotis, and G.~Vilmart.
\newblock Accelerated convergence to equilibrium and reduced asymptotic
  variance for {L}angevin dynamics using {S}tratonovich perturbations.
\newblock {\em C. R. Math. Acad. Sci. Paris}, 357(4):349--354, 2019.

\bibitem{Abdulle14hon}
A.~Abdulle, G.~Vilmart, and K.~C. Zygalakis.
\newblock High order numerical approximation of the invariant measure of
  ergodic {SDE}s.
\newblock {\em SIAM J. Numer. Anal.}, 52(4):1600--1622, 2014.

\bibitem{Bharath23sae}
K.~Bharath, A.~Lewis, A.~Sharma, and M.~V. Tretyakov.
\newblock Sampling and estimation on manifolds using the {L}angevin diffusion.
\newblock {\em arXiv preprint arXiv:2312.14882}, 2023.

\bibitem{Bogfjellmo19aso}
G.~{Bogfjellmo}.
\newblock Algebraic structure of aromatic {B}-series.
\newblock {\em J. Comput. Dyn.}, 6(2):199--222, 2019.

\bibitem{Bogfjellmo22uat}
G.~Bogfjellmo, E.~Celledoni, R.~I. McLachlan, B.~Owren, and G.~R.~W. Quispel.
\newblock Using aromas to search for preserved measures and integrals in
  {K}ahan's method.
\newblock {\em Math. Comp.}, 93(348):1633--1653, 2024.

\bibitem{Bronasco22ebs}
E.~Bronasco.
\newblock Exotic {B}-series and {S}-series: algebraic structures and order
  conditions for invariant measure sampling.
\newblock {\em Found. Comput. Math.}, pages 1--31, 2024.

\bibitem{Bronasco22cef}
E.~Bronasco and A.~Laurent.
\newblock {H}opf algebra structures for the backward error analysis of ergodic
  stochastic differential equations.
\newblock {\em Submitted}, 2024.

\bibitem{Butcher72aat}
J.~C. Butcher.
\newblock An algebraic theory of integration methods.
\newblock {\em Math. Comp.}, 26:79--106, 1972.

\bibitem{Butcher16nmf}
J.~C. Butcher.
\newblock {\em Numerical methods for ordinary differential equations}.
\newblock John Wiley \& Sons, Ltd., Chichester, third edition, 2016.

\bibitem{Butcher21bsa}
J.~C. Butcher.
\newblock {\em B-series: algebraic analysis of numerical methods}.
\newblock Springer, 2021.

\bibitem{Calaque11tih}
D.~Calaque, K.~Ebrahimi-Fard, and D.~Manchon.
\newblock Two interacting {H}opf algebras of trees: a {H}opf-algebraic approach
  to composition and substitution of {B}-series.
\newblock {\em Adv. in Appl. Math.}, 47(2):282--308, 2011.

\bibitem{Chartier10aso}
P.~Chartier, E.~Hairer, and G.~Vilmart.
\newblock Algebraic structures of {B}-series.
\newblock {\em Found. Comput. Math.}, 10(4):407--427, 2010.

\bibitem{Chartier07pfi}
P.~Chartier and A.~Murua.
\newblock Preserving first integrals and volume forms of additively split
  systems.
\newblock {\em IMA J. Numer. Anal.}, 27(2):381--405, 2007.

\bibitem{Debussche12wbe}
A.~Debussche and E.~Faou.
\newblock Weak backward error analysis for {SDE}s.
\newblock {\em SIAM J. Numer. Anal.}, 50(3):1735--1752, 2012.

\bibitem{Duncan16vru}
A.~B. Duncan, T.~Leli\`evre, and G.~A. Pavliotis.
\newblock Variance reduction using nonreversible {L}angevin samplers.
\newblock {\em J. Stat. Phys.}, 163(3):457--491, 2016.

\bibitem{Faou09csd}
E.~Faou and T.~Leli\`evre.
\newblock Conservative stochastic differential equations: mathematical and
  numerical analysis.
\newblock {\em Math. Comp.}, 78(268):2047--2074, 2009.

\bibitem{Hairer06gni}
E.~Hairer, C.~Lubich, and G.~Wanner.
\newblock {\em Geometric numerical integration}, volume~31 of {\em Springer
  Series in Computational Mathematics}.
\newblock Springer-Verlag, Berlin, second edition, 2006.
\newblock Structure-preserving algorithms for ordinary differential equations.

\bibitem{Hairer74otb}
E.~Hairer and G.~Wanner.
\newblock On the {B}utcher group and general multi-value methods.
\newblock {\em Computing (Arch. Elektron. Rechnen)}, 13(1):1--15, 1974.

\bibitem{Iserles00lgm}
A.~Iserles, H.~Z. Munthe-Kaas, S.~P. N{\o}rsett, and A.~Zanna.
\newblock Lie-group methods.
\newblock In {\em Acta numerica, 2000}, volume~9 of {\em Acta Numer.}, pages
  215--365. Cambridge Univ. Press, Cambridge, 2000.

\bibitem{Iserles07bsm}
A.~Iserles, G.~R.~W. Quispel, and P.~S.~P. Tse.
\newblock B-series methods cannot be volume-preserving.
\newblock {\em BIT Numer. Math.}, 47(2):351--378, 2007.

\bibitem{Kolar93noi}
I.~Kol\'{a}\v{r}, P.~W. Michor, and J.~Slov\'{a}k.
\newblock {\em Natural operations in differential geometry}.
\newblock Springer-Verlag, Berlin, 1993.

\bibitem{Kraft96cit}
H.~Kraft and C.~Procesi.
\newblock Classical invariant theory, a primer.
\newblock {\em Lecture Notes. Preliminary version}, 1996.

\bibitem{Kriegl97tcs}
A.~Kriegl and P.~W. Michor.
\newblock {\em The convenient setting of global analysis}, volume~53 of {\em
  Mathematical Surveys and Monographs}.
\newblock American Mathematical Society, Providence, RI, 1997.

\bibitem{Laurent21ata}
A.~Laurent.
\newblock {\em Algebraic Tools and Multiscale Methods for the Numerical
  Integration of Stochastic Evolutionary Problems}.
\newblock PhD thesis, University of Geneva, 2021.

\bibitem{Laurent23tld}
A.~Laurent.
\newblock The {L}ie derivative and {N}oether's theorem on the aromatic
  bicomplex for the study of volume-preserving numerical integrators.
\newblock {\em J. Comput. Dyn.}, 11(1):10--22, 2024.

\bibitem{Laurent23tab}
A.~Laurent, R.~I. McLachlan, H.~Z. Munthe-Kaas, and O.~Verdier.
\newblock The aromatic bicomplex for the description of divergence-free
  aromatic forms and volume-preserving integrators.
\newblock {\em Forum Math. Sigma}, 11:Paper No. e69, 2023.

\bibitem{Laurent20eab}
A.~Laurent and G.~Vilmart.
\newblock Exotic aromatic {B}-series for the study of long time integrators for
  a class of ergodic {SDE}s.
\newblock {\em Math. Comp.}, 89(321):169--202, 2020.

\bibitem{Laurent21ocf}
A.~Laurent and G.~Vilmart.
\newblock Order conditions for sampling the invariant measure of ergodic
  stochastic differential equations on manifolds.
\newblock {\em Found. Comput. Math.}, 22(3):649--695, 2022.

\bibitem{Lejay22cgr}
A.~Lejay.
\newblock Constructing general rough differential equations through flow
  approximations.
\newblock {\em Electron. J. Probab.}, 27:Paper No. 7, 24, 2022.

\bibitem{Lelievre13onr}
T.~Leli\`evre, F.~Nier, and G.~A. Pavliotis.
\newblock Optimal non-reversible linear drift for the convergence to
  equilibrium of a diffusion.
\newblock {\em J. Stat. Phys.}, 152(2):237--274, 2013.

\bibitem{Lelievre10fec}
T.~Leli\`evre, M.~Rousset, and G.~Stoltz.
\newblock {\em Free energy computations}.
\newblock Imperial College Press, London, 2010.
\newblock A mathematical perspective.

\bibitem{Linares21tsg}
P.~Linares, F.~Otto, and M.~Tempelmayr.
\newblock The structure group for quasi-linear equations via universal
  enveloping algebras.
\newblock {\em Comm. Amer. Math. Soc.}, 3:1--64, 2023.

\bibitem{Luan13ebs}
V.~T. Luan and A.~Ostermann.
\newblock Exponential {B}-series: the stiff case.
\newblock {\em SIAM J. Numer. Anal.}, 51(6):3431--3445, 2013.

\bibitem{Lundervold11hao}
A.~Lundervold and H.~Munthe-Kaas.
\newblock Hopf algebras of formal diffeomorphisms and numerical integration on
  manifolds.
\newblock In {\em Combinatorics and physics}, volume 539 of {\em Contemp.
  Math.}, pages 295--324. Amer. Math. Soc., Providence, RI, 2011.

\bibitem{Lundervold13bea}
A.~Lundervold and H.~Munthe-Kaas.
\newblock Backward error analysis and the substitution law for {L}ie group
  integrators.
\newblock {\em Found. Comput. Math.}, 13(2):161--186, 2013.

\bibitem{Markl08git}
M.~Markl.
\newblock {${\rm GL}_n$}-invariant tensors and graphs.
\newblock {\em Arch. Math. (Brno)}, 44(5):449--463, 2008.

\bibitem{McLachlan16bsm}
R.~I. McLachlan, K.~Modin, H.~Munthe-Kaas, and O.~Verdier.
\newblock B-series methods are exactly the affine equivariant methods.
\newblock {\em Numer. Math.}, 133(3):599--622, 2016.

\bibitem{McLachlan17bsa}
R.~I. McLachlan, K.~Modin, H.~Munthe-Kaas, and O.~Verdier.
\newblock Butcher series: a story of rooted trees and numerical methods for
  evolution equations.
\newblock {\em Asia Pac. Math. Newsl.}, 7(1):1--11, 2017.

\bibitem{MuntheKaas24gio}
H.~Munthe-Kaas.
\newblock Geometric integration on symmetric spaces.
\newblock {\em J. Comput. Dyn.}, 11(1):43--58, 2024.

\bibitem{MuntheKaas23lat}
H.~Munthe-Kaas and J.~Stava.
\newblock Lie admissible triple algebras: The connection algebra of symmetric
  spaces.
\newblock {\em Submitted}, 2023.

\bibitem{MuntheKaas16abs}
H.~Munthe-Kaas and O.~Verdier.
\newblock Aromatic {B}utcher series.
\newblock {\em Found. Comput. Math.}, 16(1):183--215, 2016.

\bibitem{MuntheKaas14ssa}
H.~Z. Munthe-Kaas, G.~R.~W. Quispel, and A.~Zanna.
\newblock Symmetric spaces and {L}ie triple systems in numerical analysis of
  differential equations.
\newblock {\em BIT Numer. Math.}, 54(1):257--282, 2014.

\bibitem{Rahm22aoa}
L.~Rahm.
\newblock An operadic approach to substitution in {L}ie-{B}utcher series.
\newblock {\em Forum Math. Sigma}, 10:Paper No. e20, 29, 2022.

\bibitem{Shardlow06mef}
T.~Shardlow.
\newblock Modified equations for stochastic differential equations.
\newblock {\em BIT Numer. Math.}, 46(1):111--125, 2006.

\bibitem{Talay90eot}
D.~Talay and L.~Tubaro.
\newblock Expansion of the global error for numerical schemes solving
  stochastic differential equations.
\newblock {\em Stochastic Anal. Appl.}, 8(4):483--509 (1991), 1990.

\bibitem{Weyl39tcg}
H.~Weyl.
\newblock {\em The {C}lassical {G}roups. {T}heir {I}nvariants and
  {R}epresentations}.
\newblock Princeton University Press, Princeton, N.J., 1939.

\bibitem{Zygalakis11ote}
K.~C. Zygalakis.
\newblock On the existence and the applications of modified equations for
  stochastic differential equations.
\newblock {\em SIAM J. Sci. Comput.}, 33(1):102--130, 2011.

\end{thebibliography}

%\vskip-1ex
\newpage
\begin{appendices}

\section{Exotic aromatic trees of order 3}
\label{section:list_trees_order_3}

\begin{table}[!htb]
	\setcellgapes{2pt}
	\centering
	\begin{tabular}{|c|c|c|c|c|c|c|}
	\hline
	$\abs{\kappa}$ &~$\kappa$ &~$\boldsymbol{\kappa}'$ &~$\tau$ &~$\sigma$ &~$\gamma$ &~$\FF(\gamma)(f)$ \\
	\hhline{|=|=|=|=|=|=|=|}
	$1$ &~$(0,0,0,0,1)$ &~$(0,0,0,0,4)$ &~$(1,1,1,1)$ &~$(\Circled{0},1)(\Circled{1},\Circled{2})(\Circled{3},\Circled{4})$ &~$\eatree{00001}{020}{1}$ &~$f_{jjkk}^i\partial_i$ \\
	\cline{5-7}
	 &  &  &  &~$(\Circled{0},\Circled{1})(\Circled{2},1)(\Circled{3},\Circled{4})$ &~$\eatree{00001}{120}{1}$ &~$f^j_{ijkk}\partial_i$ \\
	\hline
	$2$ &~$(0,1,1)$ &~$(0,1,2)$ &~$(1,1,2)$ &~$(\Circled{0},2)(\Circled{1},\Circled{2})(\Circled{3},1)$ &~$\eatree{011}{010}{1}$ &~$f^i_j f^j_{kk} \partial_i$ \\
	\cline{5-7}
	 &  &  &  &~$(\Circled{0},2)(\Circled{1},1)(\Circled{2},\Circled{3})$ &~$\eatree{011}{110}{1}$ &~$f^i_j f^k_{jk} \partial_i$ \\
	 \cline{5-7}
	 &  &  &  &~$(\Circled{0},\Circled{3})(1,2)(\Circled{1},\Circled{2})$ &~$\eatree{011}{021}{1}$ &~$f^j_i f^j_{kk} \partial_i$ \\
	 \cline{5-7}
	 &  &  &  &~$(\Circled{0},\Circled{3})(\Circled{1},1)(\Circled{2},2)$ &~$\eatree{011}{110}{3}$ &~$f^j_i f^k_{jk} \partial_i$ \\
	 \cline{5-7}
	 &  &  &  &~$(\Circled{0},1)(\Circled{1},2)(\Circled{2},\Circled{3})$ &~$\eatree{011}{010}{2}$ &~$f^i_{jk} f^j_k \partial_i$ \\
	 \cline{5-7}
	 &  &  &  &~$(\Circled{0},\Circled{1})(\Circled{3},1)(\Circled{2},2)$ &~$\eatree{011}{100}{1}$ &~$f^k_{ij} f^j_k \partial_i$ \\
	 \cline{5-7}
	 &  &  &  &~$(\Circled{0},\Circled{1})(1,2)(\Circled{2},\Circled{3})$ &~$\eatree{011}{021}{2}$ &~$f^j_{ik} f^j_k \partial_i$ \\
	 \cline{5-7}
	 &  &  &  &~$(\Circled{0},1)(\Circled{1},\Circled{2})(\Circled{3},2)$ &~$\eatree{011}{110}{2}$ &~$f^i_{jj} f^k_k \partial_i$ \\
	 \cline{5-7}
	 &  &  &  &~$(\Circled{0},\Circled{1})(\Circled{2},1)(\Circled{3},2)$ &~$\eatree{011}{210}{1}$ &~$f^j_{ij} f^k_k \partial_i$ \\
	\hline
	$2$ &~$(1,0,0,1)$ &~$(0,0,0,3)$ &~$(1,1,1)$ &~$(\Circled{0},1)(\Circled{1},2)(\Circled{2},\Circled{3})$ &~$\eatree{1001}{010}{1}$ &~$f^i_{jkk} f^j \partial_i$ \\
	 \cline{5-7}
	 &  &  &  &~$(\Circled{0},\Circled{1})(\Circled{2},1)(\Circled{3},2)$ &~$\eatree{1001}{110}{2}$ &~$f^j_{ijk} f^k \partial_i$ \\
	 \cline{5-7}
	 &  &  &  &~$(\Circled{0},\Circled{1})(\Circled{2},\Circled{3})(1,2)$ &~$\eatree{1001}{021}{1}$ &~$f^j_{ikk} f^j \partial_i$ \\
	\cline{5-7}
	 &  &  &  &~$(\Circled{0},2)(\Circled{1},1)(\Circled{2},\Circled{3})$ &~$\eatree{1001}{110}{1}$ &~$f^i  f^j_{jkk} \partial_i$ \\
	 \hline
	\end{tabular}
	\renewcommand\thetable{3 (Part 1/2)}
	\caption{List of the exotic aromatic trees of order three, with their associated composition, derived composition, target map, source map, and elementary differential (see Definition~\ref{definition:elementary_differential}).}
%	\label{table:exotic_aromatic_trees_order_3}
	\setcellgapes{1pt}
\end{table}

\begin{table}[!htb]
	\setcellgapes{2pt}
	\centering
	\begin{tabular}{|c|c|c|c|c|c|c|}
	\hline
	$\abs{\kappa}$ &~$\kappa$ &~$\boldsymbol{\kappa}'$ &~$\tau$ &~$\sigma$ &~$\gamma$ &~$\FF(\gamma)(f)$ \\
	\hhline{|=|=|=|=|=|=|=|}
	~$3$ &~$(1,2)$ &~$(0,2)$ &~$(1,2)$ &~$(\Circled{0},1)(\Circled{1},2)(\Circled{2},3)$ & \eatree{12}{000}{1} &~$f^i_j f^j_k f^k \partial_i$ \\
	 \cline{5-7}
	 &  &  &  &~$(\Circled{0},1)(\Circled{1},\Circled{2})(2,3)$ & \eatree{12}{011}{2} &~$f^i_j f^k_j f^k \partial_i$ \\
	 \cline{5-7}
	 &  &  &  &~$(\Circled{0},\Circled{1})(1,2)(\Circled{2},3)$ & \eatree{12}{011}{1} &~$f^j_i f^j_k f^k \partial_i$ \\
	 \cline{5-7}
	 &  &  &  &~$(\Circled{0},\Circled{1})(\Circled{2},1)(2,3)$ & \eatree{12}{011}{4} &~$f^j_i f^k_j f^k \partial_i$ \\
	\cline{5-7}
	 &  &  &  &~$(\Circled{0},1)(\Circled{1},3)(\Circled{2},2)$ & \eatree{12}{100}{1} &~$f^i_j f^j f^k_k \partial_i$ \\
	 \cline{5-7}
	 &  &  &  &~$(\Circled{0},\Circled{2})(\Circled{1},1)(2,3)$ & \eatree{12}{111}{1} &~$f^j_i f^j f^k_k \partial_i$ \\
	\cline{5-7}
	 &  &  &  &~$(\Circled{0},3)(\Circled{1},2)(\Circled{2},1)$ & \eatree{12}{100}{2} &~$f^i f^j_k f^k_j \partial_i$ \\
	 \cline{5-7}
	 &  &  &  &~$(\Circled{0},3)(\Circled{1},\Circled{2})(1,2)$ & \eatree{12}{011}{3} &~$f^i f^j_k f^j_k \partial_i$ \\
	 \cline{5-7}
	 &  &  &  &~$(\Circled{0},3)(\Circled{1},1)(\Circled{2},2)$ & \eatree{12}{200}{1} &~$f^i f^j_j f^k_k \partial_i$ \\
	 \hline
	$3$ &~$(2,0,1)$ &~$(0,0,2)$ &~$(1,1)$ &~$(\Circled{0},1)(\Circled{1},2)(\Circled{2},3)$ &~$\eatree{201}{000}{1}$ &~$f^i_{jj} f^j f^j \partial_i$ \\
	 \cline{5-7}
	 &  &  &  &~$(\Circled{0},\Circled{1})(\Circled{2},2)(1,3)$ &~$\eatree{201}{011}{3}$ &~$f^j_{ik} f^j f^k \partial_i$ \\
	 \cline{5-7}
	 &  &  &  &~$(\Circled{0},1)(\Circled{1},\Circled{2})(2,3)$ &~$\eatree{201}{011}{1}$ &~$f^i_{jj} f^k f^k \partial_i$ \\
	 \cline{5-7}
	 &  &  &  &~$(\Circled{0},\Circled{1})(\Circled{2},1)(2,3)$ &~$\eatree{201}{111}{1}$ &~$f^j_{ij} f^k f^k \partial_i$ \\
	\cline{5-7}
	 &  &  &  &~$(\Circled{0},2)(\Circled{1},1)(\Circled{2},3)$ &~$\eatree{201}{100}{1}$ &~$f^i f^j_{jk} f^k \partial_i$ \\
	 \cline{5-7}
	 &  &  &  &~$(\Circled{0},3)(\Circled{1},\Circled{2})(1,2)$ &~$\eatree{201}{011}{2}$ &~$f^i f^j f^j_{kk} \partial_i$ \\
	 \hline
	~$4$ &~$(3,1)$ &~$(0,1)$ &~$(1)$ &~$(\Circled{0},1)(\Circled{1},2)(3,4)$ & \eatree{31}{001}{2} &~$f^i_j f^j f^k f^k \partial_i$ \\
	\cline{5-7}
	 &  &  &  &~$(\Circled{0},4)(1,2)(\Circled{1},3)$ & \eatree{31}{001}{1} &~$f^i f^j f^j_k f^k \partial_i$ \\
	 \cline{5-7}
	 &  &  &  &~$(\Circled{0},\Circled{1})(1,2)(3,4)$ & \eatree{31}{012}{1} &~$f^j_i f^j f^k f^k \partial_i$ \\
	 \cline{5-7}
	 &  &  &  &~$(\Circled{0},2)(\Circled{1},1)(3,4)$ & \eatree{31}{101}{1} &~$f^i f^j f^j f^k_k \partial_i$ \\
	 \hline
	~$5$ &~$(5)$ &~$(0)$ &  &~$(\Circled{0},1)(2,3)(4,5)$ & \eatree{5}{002}{1} &~$f^i f^j f^j f^k f^k \partial_i$ \\
	 \hline
	\end{tabular}
	\renewcommand\thetable{3 (Part 2/2)}
	\caption{List of the exotic aromatic trees of order three, with their associated composition, derived composition, target map, source map, and elementary differential (see Definition~\ref{definition:elementary_differential}).}
	\setcellgapes{1pt}
\end{table}

\end{appendices}

\end{document}